\documentclass[12pt]{amsart}
\usepackage{amsmath,amsthm,amscd,amssymb,amsfonts}
\usepackage{mathrsfs,upgreek}
\usepackage{hyperref}
\usepackage{mathtools}
\usepackage{graphicx}
\usepackage{fullpage}
\usepackage{cleveref}
\usepackage{cleveref}
\usepackage[switch]{lineno}
\usepackage{etoolbox}
\usepackage{enumitem}
\modulolinenumbers[1]

\makeatletter
\let\@wraptoccontribs\wraptoccontribs
\makeatother

\usepackage[sans]{dsfont} 
\usepackage[T1]{fontenc}
\DeclareMathAlphabet{\mathpzc}{OT1}{pzc}{m}{it} 

\numberwithin{equation}{section} 
\numberwithin{figure}{section} 

\theoremstyle{plain}
\newtheorem{theo}{Theorem}
\newtheorem{prop}{Proposition}[section]
\newtheorem{coro}[prop]{Corollary}
\newtheorem{lemm}[prop]{Lemma}

\newtheorem{theoalph}{Theorem}

\theoremstyle{definition}
\newtheorem{defi}[prop]{Definition}

\theoremstyle{remark}
\newtheorem{rema}[prop]{Remark}
\newtheorem{exam}[prop]{Example}

\newtheoremstyle{citing}
  {3pt}
  {3pt}
  {\itshape}
  {}
  {\bfseries}
  {.}
  {.5em}
  {\thmnote{#3}}

\theoremstyle{citing}

\newcommand{\C}{\mathbb{C}}

\newcommand{\N}{\mathbb{N}}

\newcommand{\R}{\mathbb{R}}

\newcommand{\sB}{\mathscr{B}}

\newcommand{\sF}{\mathscr{F}}

\newcommand{\teta}{\widetilde{\teta}}

\renewcommand{\=}{\coloneqq}
\newcommand{\dd}{\hspace{1pt}\operatorname{d}\hspace{-1pt}}

\DeclareMathOperator{\diam}{diam}

\newcommand{\Nmax}{N_\text{max}(f)}

\begin{document}

\title[Thermodynamic formalism for intermittent maps with multiple fixed points and phase transitions]{Thermodynamic formalism for intermittent maps with multiple neutral fixed points and phase transitions}
\author{Daniel Coronel}
 \address{Facultad de Matem\'aticas, Pontificia Universidad Cat\'olica de Chile, Campus San Joaqu\'in, Avenida Vicu\~{n}a Mackenna 4860, 
Santiago, Chile.} 
\email{acoronel@uc.cl}
\author{Francisco Monardes}
\address{Department of Mathematics, The Pennsylvania State University, University Park, PA 16802, USA.}
\email{fam5364@psu.edu}

\begin{abstract}
In this paper, we develop the thermodynamic formalism for intermittent full-branch interval maps with finitely many neutral fixed points and for a broad class of possibly discontinuous potentials. Under a hyperbolicity assumption on the potential, we prove the existence of a unique equilibrium state with positive entropy and exponential decay of correlations. Furthermore, we characterize the phase transitions for the one-parameter family $\beta\varphi$, with $\beta>0$, and prove that, at the critical parameter, multiple ergodic equilibrium states may coexist, at most one of which has positive entropy. Our approach relies on a detailed analysis of conformal measures.
\end{abstract}

\maketitle

%
%

\section{Introduction}
\label{s:intro}
The space of invariant measures of a dynamical system is typically large, and identifying the physically or statistically relevant ones is a central problem in ergodic theory. Thermodynamic formalism provides a powerful framework for selecting such measures. In the uniformly hyperbolic setting, it is classical that every Hölder-continuous potential admits a unique equilibrium state with strong statistical properties, and the associated pressure function is real analytic.

Beyond the uniformly hyperbolic regime, the situation becomes significantly more complex. For non-uniformly hyperbolic systems, Hölder-continuous potentials may admit multiple equilibrium states, and the pressure function may fail to be real analytic. This phenomenon is known as a phase transition. One of its simplest manifestations occurs in the Manneville--Pomeau family \cite{MannevillePomeau}, where the presence of a neutral fixed point breaks uniform expansiveness and can lead to phase transitions even for Hölder-continuous potentials (see \cite{BruinTerhesiuTodd,ClimenhagaThompson,coronelRivera1,coronelRivera2,PrellbergSlawny,Sarig1}).

The presence of neutral fixed points is commonly referred to as intermittency. This phenomenon has been extensively studied, with much of the literature focusing on the Manneville--Pomeau family and its variants, which are characterized by having a single source of intermittency (see, for example, \cite{BaladiTodd,BruinTerhesiuTodd,CastroVarandas,ClimenhagaThompson,coronelRivera1,coronelRivera2,GaribaldiInoquio1,GaribaldiInoquio2,Gouzel2,Holland,LiRivera2,Liverani,MelbourneTerhesiu,PollicottWeiss,PrellbergSlawny,Sarig1,Sarig2}). Maps with multiple neutral fixed points have also been considered (\cite{MelbourneTerhesiu,Thaler1,Thaler2,Thaler3,zweimuller1,zweimuller2}), and have recently attracted renewed attention (\cite{coatesGelfert,coatesMelbourne,CoatesMelbourneTalebi,coatesLuzzatto,coatesLuzzattoMubarak}).

In this paper, we study equilibrium states and phase transitions for a class of intermittent interval maps with finitely many branches and multiple neutral fixed points. We consider a class of branch Hölder-continuous potentials (see Definition \ref{branch Holder}), which are designed to retain key features of Hölder regularity while allowing for discontinuities.

We prove that a subclass of these potentials satisfies a hyperbolicity condition under which there exists a unique equilibrium state with strong statistical properties, closely resembling the uniformly hyperbolic case (see Theorem \ref{general theo A}). On the other hand, we show that branch Hölder-continuous potentials can exhibit at most one phase transition, beyond which the pressure function becomes affine. At the transition parameter, multiple ergodic equilibrium states may coexist; however, at most one of them has positive entropy (see Theorem \ref{Theo A}).

A key ingredient in our analysis is a detailed description of conformal measures, given in Theorem \ref{t: conformal measure}, which may be of independent interest.

An important advantage of allowing discontinuous potentials in our setting is that our results apply to the family of geometric potentials. This case is studied in detail in a companion paper \cite{CoronelMonardes}, where a complete description of the corresponding phase transition is obtained.

\subsection{Full branch maps with neutral fixed points}

\paragraph{\textbf{Notation.}} Given functions $g$ and $h$, we write $g(x)\sim h(x)$ as $x\to x_0$ if $g(x)/h(x)\to 1$ as $x\to x_0$.

We define $\sF$ to be the class of interval maps $f\colon[0,1]\to[0,1]$ satisfying the following properties.

\begin{enumerate}[label=\textbf{P\arabic*.}, ref=P\arabic*]
\item\label{itm P1} There exist subintervals $I_1,\ldots,I_d$, with $d\geq 2$, such that $\bigcup_{k=1}^d I_k = [0,1]$, and for each $k$, the restriction $f|_{I_k}$ is a bijection of class $C^1$ with Hölder-continuous derivative, admits a homeomorphic extension $f_k\colon \overline{I}_k \to [0,1]$, and the derivative $Df|_{\text{int}I_k}$ admits a continuous extension to $\overline{I}_k$.

\item\label{itm P2} There exists a nonempty subset $N(f)\subset \{1,\ldots,d\}$ such that for every $k\in N(f)$ the following hold:
\begin{itemize}
\item[(1)] There exists $\xi_k\in I_k$ such that $f(\xi_k)=\xi_k$, $Df(\xi_k)=1$, and $Df(x)>1$ for every $x\in I_k\setminus\{\xi_k\}$.

\item[(2)] There exist constants $\alpha_k>0$ and $b_k>0$ such that
\begin{equation}
|f(x)-x| \sim b_k |x-\xi_k|^{1+\alpha_k} \quad \text{as } x\to\xi_k,
\end{equation}
\item[(3)] The restriction $\log Df|_{I_k}$ is Hölder-continuous of exponent $\min\{1,\alpha_k\}$.
\end{itemize}

\item\label{itm P3} For each $k\in \{1,\ldots,d\}\setminus N(f)$, we have 
$$
\inf_{x\in I_k}Df(x)>1.
$$
\end{enumerate}

For a map $f\in \mathscr{F}$ and each $k\in \{1,\ldots,d\}$, we denote by $\xi_k$ the fixed point of $f$ contained in $I_k$, and we define
\[
E(f) := \{1,\ldots,d\}\setminus N(f).
\]
If $k\in N(f)$, we say that $\xi_k$ is a \emph{neutral fixed point}.
We also define $0= x_0 <x_1<\dots<x_d=1$  so that $\text{int } I_k =(x_{k-1},x_k)$ for every $k\in\{1,\dots,d\}$.
\subsection{Keller spaces}\label{subseq: Keller spaces}

In this subsection, we recall a function space introduced in \cite{Keller1}. We follow \cite{Juan1}, where a more detailed exposition can be found.

Let $I$ be a compact interval in $\mathbb{R}$, and let $m$ be an atom-free Borel probability measure on $I$. We consider functions defined on $I$ up to equality almost everywhere with respect to $m$.

We define a pseudometric $d$ on $I$ by
\begin{equation*}
    d(x,y) := m(\{z\in I : x\leq z\leq y \text{ or } y\leq z \leq x\}),
\end{equation*}
and for $x\in I$ and $\varepsilon>0$, we define
\begin{equation*}
    B_d(x,\varepsilon) := \{y\in I : d(x,y)<\varepsilon\}.
\end{equation*}

Given a function $h\colon I \to \mathbb{C}$ and $\varepsilon>0$, for each $x\in I$ we define
\begin{equation*}
    \operatorname{osc}(h,\varepsilon,x) := \operatorname*{ess\,sup}\{|h(y)-h(y')| : y,y' \in B_d(x,\varepsilon)\},
\end{equation*}
and
\begin{equation*}
    \operatorname{osc}_1(h,\varepsilon) := \int_{I}\operatorname{osc}(h,\varepsilon,x) \dd m(x).
\end{equation*}

Fix $A>0$. For $\gamma\in(0,1]$ and $h\colon I\to\mathbb{C}$, we define
\begin{equation*}
    |h|_{\gamma,1} := \sup_{\varepsilon\in (0,A]} \frac{\operatorname{osc}_1(h,\varepsilon)}{\varepsilon^{\gamma}}, 
    \quad 
    \| h\|_{\gamma,1} := \| h\|_1 + |h|_{\gamma,1}.
\end{equation*}

Observe that $|h|_{\gamma,1}$ and $\|h\|_{\gamma,1}$ depend only on the equivalence class of $h$. Let $H^{\gamma,1}(m)$ be the space of equivalence classes of functions $h\colon I\to\mathbb{C}$ such that $\|h\|_{\gamma,1}<+\infty$. The quantity $\|\cdot\|_{\gamma,1}$ defines a norm on $H^{\gamma,1}(m)$.

Keller proved in \cite[Theorem 1.13]{Keller1} that $H^{\gamma,1}(m)$ is a Banach space with respect to this norm.

\subsection{Statement of results}

Let $f:[0,1]\to[0,1]$ be a map in $\sF$. A \emph{potential} is a Borel measurable function $\varphi\colon [0,1]\to\R$, and for every $n\in \N$ we define
\begin{equation*}
    S_n\varphi := \sum_{j=0}^{n-1} \varphi\circ f^j.
\end{equation*}
We denote by $\mathcal{M}_f$ the space of $f$-invariant Borel probability measures. The \emph{pressure} of a potential $\varphi$ is defined by
\begin{equation}\label{pressure}
    P(\varphi) := \sup\left\{h_{\mu}(f) + \int \varphi\dd\mu : \mu\in \mathcal{M}_f\right\},
\end{equation}
where $h_{\mu}(f)$ denotes the measure-theoretic entropy of $\mu$. A measure attaining the supremum in \eqref{pressure} is called an \emph{equilibrium state} for $\varphi$.

Given two Borel measurable partitions $\mathcal{P}$ and $\mathcal{Q}$ of $[0,1]$, we define
\begin{equation*}
    f^{-1}\mathcal{P} := \{f^{-1}(P): P \in \mathcal{P}\}
    \quad \text{and} \quad
    \mathcal{P}\vee \mathcal{Q} := \{P\cap Q : P\in \mathcal{P},\; Q\in\mathcal{Q}\}.
\end{equation*}

\begin{defi}\label{branch Holder}
Let $\gamma\in (0,1]$ and let $\mathcal{P}:=\{I_1,\ldots,I_d\}$. We say that a potential $\varphi\colon [0,1]\to\R$ is \emph{branch Hölder-continuous of exponent} $\gamma$ (resp. \emph{branch continuous}) \emph{with respect to} $f$ if there exists $N\in \N$ such that for every
\begin{equation*}
P\in \bigvee_{j=0}^{N-1}f^{-j}\mathcal{P},
\end{equation*}
the restriction $\varphi|_P$ is Hölder-continuous of exponent $\gamma$ (resp. continuous) and admits a continuous extension to the boundary of $P$. We denote by $BH^{\gamma}(f)$ the space of all branch Hölder-continuous potentials of exponent $\gamma$ with respect to $f$.
\end{defi}

Given a Borel measurable function $g\colon [0,1]\to [0,+\infty)$, a Borel probability measure $\nu$ on $[0,1]$ is called \emph{$g$-conformal} for $f$ if for every Borel measurable set $A\subset [0,1]$ on which $f$ is injective we have
\begin{equation*}
    \nu(f(A)) = \int_A g \dd\nu.
\end{equation*}

We say that $\varphi\colon[0,1]\to\R$ is \emph{hyperbolic} for $f$ if there exists $n\in \N$ such that
\begin{equation}\label{equation hyperbolicity introduction}
    \sup_{x\in[0,1]}\frac{1}{n}S_n\varphi(x) < P(\varphi).
\end{equation}
In Proposition \ref{prop characterization of hyperbolicity} we prove that a potential $\varphi\in BH^{\gamma}(f)$ is hyperbolic if and only if
\begin{equation}\label{eq equivalence of hyperbolicity}
    P(\varphi)> \max_{k\in N(f)}\varphi(\xi_k).
\end{equation}

A neutral fixed point $\xi$ is called \emph{maximizing for the potential $\varphi$} if
\begin{equation}
    \varphi(\xi) = \max_{k\in N(f)}\varphi(\xi_k).
\end{equation}
We denote by $\Nmax\subseteq N(f)$ the set of indices of maximizing neutral fixed points.

\begin{theoalph}\label{t: conformal measure}
Let $\varphi$ be a branch continuous potential with respect to $f$. There exists an $\exp(P(\varphi)-\varphi)$-conformal measure for $f$. Moreover, the following hold.
\begin{enumerate}
    \item[1.] If $\varphi$ is hyperbolic for $f$, then every  $\exp(P(\varphi)-\varphi)$-conformal measure is atom-free.
    \item[2.] There is at most one $\exp(P(\varphi)-\varphi)$-conformal measure that is not fully supported. When such a measure exists, it is the Dirac measure at a maximizing neutral fixed point $\xi\in\{0,1\}$ such that $f^{-1}(\xi)=\{\xi\}$.
    \item[3.] For $\xi\in\{0,1\}$, if $\delta_\xi$ is the unique  $\exp(P(\varphi)- \varphi)$-conformal measure for $f$, then
    \begin{equation}\label{e: eq 1 thm A}
        \sum_{n=0}^{+\infty}\sum_{x\in f^{-n}(\{x_1,\dots,x_{d-1}\})}\exp(S_n\varphi(x)- nP(\varphi))<+\infty.
    \end{equation}
\end{enumerate}
\end{theoalph}

Given $\varphi\in BH^{\gamma}(f)$ for some $\gamma\in(0,1]$, we define the \emph{transfer operator} $\mathcal{L}_{\varphi}$ acting on the space of bounded complex-valued functions on $[0,1]$ by
\begin{equation}\label{eq transfer operator 3}
    \mathcal{L}_{\varphi}h (x) := \sum_{y\in f^{-1}(x)}\exp(\varphi(y))h(y).
\end{equation}

\begin{theoalph}\label{general theo A}
Let $\gamma\in(0,1]$, and let $\varphi \in BH^{\gamma}(f)$ be hyperbolic for $f$. Let $m$ be the atom-free $\exp(P(\varphi)-\varphi)$-conformal measure given by Theorem \ref{t: conformal measure}. Then there exists $A>0$ such that, for the space $H^{\gamma,1}(m)$ defined in Subsection \ref{subseq: Keller spaces} with $I=[0,1]$, the following properties hold.
\begin{enumerate}
    \item[1.] The operator $\mathcal{L}_{\varphi}$ maps $H^{\gamma,1}(m)$ to itself, and $\mathcal{L}_{\varphi}|_{H^{\gamma,1}(m)}$ is bounded. Moreover, the number $\exp(P(\varphi))$ is an eigenvalue of algebraic multiplicity $1$ of $\mathcal{L}_{\varphi}|_{H^{\gamma,1}(m)}$, and there exists $\rho\in (0,1)$ such that the spectrum of $\mathcal{L}_{\varphi}|_{H^{\gamma,1}(m)}$ is contained in $B(0,\rho\exp(P(\varphi)))\cup \{\exp(P(\varphi))\}$.

    \item[2.] There exists a unique equilibrium state $\mu$ for the potential $\varphi$. Moreover, this measure is absolutely continuous with respect to $m$, and its entropy is strictly positive. Finally, there exists a constant $C>0$ such that for every integer $n\geq 1$, every bounded measurable function $\phi\colon[0,1]\to \C$, and every $\psi\in H^{\gamma,1}(m)$, we have
    \begin{equation}\label{eq exp decay}
        \left|\int \phi\circ f^n \cdot \psi \dd\mu-\int \phi \dd\mu \int \psi \dd\mu\right|
        \leq C\lVert \phi\rVert_{\infty} \lVert \psi \rVert_{\gamma,1} \rho^n.
    \end{equation}

    \item[3.] For every branch Hölder-continuous function $\chi\colon [0,1]\to \R$ of exponent $\overline{\gamma}\in (0,1]$, the function $t\mapsto P(\varphi +t \chi)$ is real analytic on a neighborhood of $t=0$.
\end{enumerate}
\end{theoalph}

\begin{rema}
The behavior of the transfer operator $\mathcal{L}_\varphi$ described in part 1 of Theorem \ref{general theo A} is known as a \emph{spectral gap} and has strong consequences. Besides exponential decay of correlations, other stochastic properties, such as the central limit theorem and the almost sure invariance principle, can also be deduced. For further details, see \cite[Theorem 3.3]{Keller1}.
\end{rema}

The following is the main result of the paper. It gives a description of the possible phase transitions for the family $\beta \varphi$: either the family remains hyperbolic for all $\beta>0$, or there is a critical parameter $\beta_{*}$ after which the pressure becomes affine and the equilibrium states are supported on maximizing neutral fixed points. Here, $|\cdot|$ denotes cardinality.

\begin{theoalph}\label{Theo A}
Let $\gamma\in(0,1]$ and let $\varphi\in BH^\gamma(f)$. The following dichotomy holds.
\begin{enumerate}
    \item[1.] Either $\beta\varphi$ is hyperbolic for every $\beta>0$, and hence the conclusions of Theorem \ref{general theo A} hold for every $\beta>0$;

    \item[2.] or there exists $\beta_0>0$ such that
    \[
    P(\beta_0\varphi)=\beta_0\max_{k\in N(f)}\varphi(\xi_k),
    \]
    and for
    \begin{equation}\label{eq 1 thm C}
        \beta_* := \inf\left\{\beta>0: P(\beta\varphi) = \beta\max_{k\in N(f)}\varphi(\xi_k)\right\},
    \end{equation}
    the following hold.
    \begin{enumerate}
        \item[(a)] For every $\beta\in(0,\beta_*)$, the potential $\beta\varphi$ is hyperbolic.

        \item[(b)] For every $\beta\geq \beta_*$, we have
        \begin{equation}\label{eq 2 thm 1}
            P(\beta\varphi)=\beta\max_{k\in N(f)}\varphi(\xi_k).
        \end{equation}
        Furthermore,
        \begin{enumerate}
            \item[(i)] the potential $\beta_*\varphi$ admits at most $|\Nmax|+1$ ergodic equilibrium states, $|\Nmax|$ of which are supported on maximizing neutral fixed points, while the possible additional one has positive entropy, is absolutely continuous with respect to a fully supported $\exp(P(\beta_*\varphi)-\beta_{*}\varphi)$-conformal measure and its density is bounded from below by a positive constant;

            \item[(ii)] for every $\beta>\beta_*$, the potential $\beta\varphi$ has exactly $|\Nmax|$ ergodic equilibrium states, each supported on a maximizing neutral fixed point.
        \end{enumerate}

        \item[(c)] The pressure function $\beta\mapsto P(\beta\varphi)$ is real analytic on $(0,+\infty)\setminus\{\beta_*\}$.
    \end{enumerate}
\end{enumerate}
\end{theoalph}
\begin{exam}
    Let $\alpha_1,\alpha_3>0$ and consider the map $f\colon[0,1]\to [0,1]$ given by
    \begin{equation}
        f(x) \= \begin{cases}
            x+ 2\cdot 3^{\alpha_1}x^{1+\alpha_1}, & x\in [0,1/3],\\[3pt]
            3x-1, & x\in(1/3,2/3),\\[3pt]
            x-2\cdot 3^{\alpha_3}(1-x)^{1+\alpha_3}, & x\in [2/3,1].
        \end{cases} 
    \end{equation}
    The map $f$ belongs to $\mathscr{F}$, with 
    \begin{equation*}
        I_1 = [0,1/3],\; I_2=(1/3,2/3),\; I_3 = [2/3,1],
    \end{equation*}
    and
    \begin{equation*}
        N(f) = \{1,3\},\; E(f)=\{2\}.
    \end{equation*}
     The neutral fixed points are $\xi_1=0$ and $\xi_3=1$, and $\xi_2=1/2$ is a fixed point corresponding to a uniformly expanding branch. 
     
     Let $\gamma_1,\gamma_2$, and $\gamma_3$ be positive numbers. Define $\varphi\colon[0,1]\to \R$ by 
    \begin{equation*}
        \varphi(x) \= -x^{\gamma_1}|x-1/2|^{\gamma_2}(1-x)^{\gamma_3}.
    \end{equation*}
   The function $\varphi$ is Hölder-continuous of exponent $\gamma\= \min\{1,\gamma_1,\gamma_2,\gamma_3\}$ and hence belongs to $BH^\gamma(f)$. By Lemma \ref{key lemma},
    \begin{equation*}
        P(\beta\varphi) > \int \beta\varphi \dd \delta_{1/2} = \beta\varphi(1/2)=0 \text{ for every }\beta>0.
    \end{equation*}
   Since $\sup_{x\in[0,1]}\varphi(x) =0$, we have
   \begin{equation*}
       P(\beta\varphi) > \sup_{x\in[0,1]} \beta \varphi(x) \text{ for every }\beta>0.
   \end{equation*}
   That is, for each $\beta>0$ the potential $\beta\varphi$ is hyperbolic for $f$, and therefore it verifies part 1 of Theorem \ref{Theo A}.
    
    Now, consider the geometric potential $-\log Df$, where $Df(1/3)$ and $Df(2/3)$ denote the left and right derivatives, respectively. The geometric potential belongs to $BH^{\min\{1,\alpha_1,\alpha_3\}}(f)$. In  \cite{CoronelMonardes} we show that
    \begin{equation*}
         P(-\beta\log Df) > 0 \text{ for }\beta\in(0,1),
    \end{equation*}
    and
    \begin{equation*}
         P(-\beta\log Df) =0 \text{ for }\beta\in [1,+\infty). 
    \end{equation*}
    In particular,
    \begin{equation*}
        P(-\log Df) = -\log Df(0) = -\log Df(1).
    \end{equation*}
    Thus, the geometric potential satisfies condition 2 in Theorem \ref{Theo A} with $\beta_* = 1$.
    Moreover, we show that $-\log Df$ admits an equilibrium state with positive entropy if and only if $\max\{\alpha_1,\alpha_3\} \in (0,1)$.
\end{exam}

\subsection{Outline of the paper}

The remainder of the paper is organized as follows.

Since the intermittent maps and potentials considered here are allowed to be discontinuous, we first develop some machinery that enables us to apply standard results from the theory of continuous dynamical systems on compact spaces. To this end, Section \ref{sec: geometric estimates and top pressure} begins by establishing several geometric estimates for maps in the family $\sF$ (Section \ref{subseq: Basic estimates}). In Section \ref{subseq: continuous extension} we construct a continuous extension of our system (Lemma \ref{extension lemma}). This extension is the main tool used to handle discontinuities in both the map and the potential. It follows a well-known procedure described in \cite[Appendix A.5]{Keller2}, which was also applied in \cite{coronelRivera1} to the Manneville--Pomeau family in a manner similar to the one used here. Using this extension, we obtain a topological characterization of the pressure of branch continuous potentials (Lemma \ref{tree pressure}).

\Cref{sec: proof of theo A} is devoted to the proof of \Cref{t: conformal measure}. Given a branch continuous potential $\varphi$, let $\widetilde{\varphi}$ denote its extension constructed in Section \ref{sec: geometric estimates and top pressure}. We first construct an $\exp(P(\widetilde{\varphi},\widetilde{f})-\widetilde{\varphi})$-conformal measure. We then describe a mechanism for constructing atomic $\exp(P(\varphi)-\varphi)$-conformal measures when \eqref{eq equivalence of hyperbolicity} fails. Finally, we show that atomic conformal measures are incompatible with hyperbolic potentials. Combining these results yields the proof of Theorem \ref{t: conformal measure}.

\Cref{sec: Keller spaces} collects several properties of the Banach spaces introduced in Section \ref{subseq: Keller spaces}. The main results are Theorem \ref{keller theo} and Corollary \ref{corollary keller theo}, proved in \cite{Keller1} and \cite{Juan1}, respectively. These results show that the transfer operator has a spectral gap, which implies the existence of an equilibrium state with exponential decay of correlations, as well as analyticity properties of the pressure. We refer to \cite{Keller1} and \cite{Juan1} for proofs and further details.

In \Cref{sec: Proof of Theo B and C} we prove Theorems \ref{general theo A} and \ref{Theo A}. The former follows from the results in Section \ref{sec: Keller spaces}, together with \cite[Theorem 6]{Dobbs}; our argument follows the same general structure as the proof of \cite[Theorem B]{Juan1}, with only minor modifications. Unlike \cite{Juan1}, where the potential is replaced by a cohomologous averaged potential in order to reduce to the condition $\sup\varphi<P(\varphi)$, we work directly with \eqref{equation hyperbolicity introduction}. The proof of \Cref{Theo A} proceeds as follows. We show that condition \eqref{eq equivalence of hyperbolicity} is equivalent to hyperbolicity of $\varphi$, so parts 1 and 2(a) follow directly from \Cref{general theo A}. The identity \eqref{eq 2 thm 1} for $\beta\geq \beta_*$, as well as the existence of $|\Nmax|$ ergodic equilibrium states supported on maximizing neutral fixed points, are direct consequences of \eqref{pressure}. The Key Lemma (Lemma \ref{key lemma}) implies that any other ergodic equilibrium state must have positive entropy. The existence of a fully supported $\exp(P(\beta_*\varphi)-\beta_*\varphi)$-conformal measure given by Proposition \ref{p: before proof thm C} then allows us to apply \cite[Theorem 6]{Dobbs} and conclude that such an equilibrium state must be equivalent to such $\exp(P(\beta_*\varphi)-\beta_*\varphi)$-conformal measure. This yields part 2(b). Finally, part 2(c) follows from part 3 of \Cref{general theo A} together with \eqref{eq 2 thm 1}.

The paper concludes with the proof of the Key Lemma (Lemma \ref{key lemma}) in Appendix \ref{sec: key lemma}.

\subsection{Notes and comparison with related literature}
The class of maps considered here overlaps nontrivially with the interval maps studied in \cite{zweimuller1,zweimuller2,MelbourneTerhesiu}. These works focus on the infinite-measure regime and investigate the structure and statistical properties of absolutely continuous invariant measures. By contrast, we study the thermodynamic formalism for a broad class of branch Hölder-continuous potentials and the associated probability equilibrium states.

The study of interval maps with multiple neutral fixed points has seen significant recent progress (\cite{coatesGelfert,coatesMelbourne,CoatesMelbourneTalebi,coatesLuzzatto,coatesLuzzattoMubarak}). We would like to emphasize that, despite some similarities, our setting differs substantially from those considered in these recent works.

One major difference concerns the regularity of the branches. While most works assume $C^2$ regularity, we only require the branches to be of class $C^{1}$ with Hölder-continuous derivative. Another difference concerns both the number of branches and the behavior near the neutral fixed points. We allow an arbitrary finite number of branches, and we do not assume that all neutral fixed points have the same degree of ``stickiness'', in contrast with the Thaler maps considered in \cite{CoatesMelbourneTalebi} and the maps studied in \cite{coatesGelfert}. Although \cite{coatesLuzzatto} and \cite{coatesLuzzattoMubarak} allow different degrees of stickiness at neutral fixed points, they only consider maps with two branches.

We also allow for the coexistence of uniformly expanding and intermittent branches. This gives a different level of flexibility from the Thaler maps considered in \cite{CoatesMelbourneTalebi}, where all branches are of an intermittent nature.

On the other hand, one of the main limitations of our approach concerns the behavior near the discontinuity points. We require the derivative of our maps to admit continuous extensions at the discontinuity points and to be at least 1. This is a loss of generality compared to \cite{coatesLuzzatto} and \cite{coatesLuzzattoMubarak}, where the derivative is allowed to tend to zero or to infinity near the discontinuities.

We note that our results are known when $f$ belongs to the Manneville--Pomeau family and the potential $\varphi$ is Hölder-continuous. More precisely, Theorems \ref{t: conformal measure} and \ref{general theo A} follow from Theorems A.1 and A.2 in \cite{InoquioRivera2}, respectively, while part of Theorem \ref{Theo A} follows from \cite[Proposition 2.10]{coronelRivera1}.

Finally, a crucial ingredient in some of our arguments is the Key Lemma (Lemma \ref{key lemma}). This result has been proved and used in several settings. It was first established in \cite{InoquioRivera} for rational maps and later in \cite{LiHuaibin} for multimodal maps. More recently, it was adapted to the Manneville--Pomeau family in \cite{coronelRivera1}. Since none of these proofs applies directly to our setting, we provide a proof in Appendix \ref{sec: key lemma}.

\section{Geometric estimates and topological pressure}\label{sec: geometric estimates and top pressure}
In this section, we fix a map $f\in\sF$, and for every $k$ in $\{1,\ldots,d\}$ we denote by $h_k$ the continuous extension of $(f|_{I_k})^{-1}$ to the interval $[0,1]$. Recall that $0=x_0<\dots<x_d=1$ were chosen so that
\begin{equation}
    \text{int }I_k=(x_{k-1},x_{k}) \text{ for every } k\in\{1,\ldots,d\}.
\end{equation}
For $k\in \{1,\ldots,d-1\}$, define
\begin{equation}\label{e: defi x_n^k}
    x_n^k\= h_k^n(1) \text{ for every }n\in \N,
\end{equation}
and for $k\in\{2,\ldots,d\}$, define
\begin{equation}\label{e: defi y_n^k}
    y_n^k \= h_k^n (0) \text{ for every }n\in\N.
\end{equation}
It follows that for each $n\in\N$,
\begin{equation*}
    h_1^n([0,1]) = [0,x_{n}^1]\text{ , }h_d^n([0,1]) = [y_{n}^d,1] \text{ and } h_k^n([0,1]) = [y_{n}^k , x_n^k] \text{ for all } k\in\{2,\ldots,d-1\}.
\end{equation*}
Observe that for each $k\in\{1,\ldots,d-1\}$ the sequence $(x_{n}^k)_{n\in\N}$ is decreasing and 
\begin{equation}\label{e: f(x_{n+1}) = x_{n}}
f(x_{n+1}^k)= x_n^k \text{ for every }n\in\N,
\end{equation}
and for each  $k\in\{2,\ldots,d\}$ the sequence $ (y_{n}^k)_{n\in \N}$ is increasing, and we have
\begin{equation}\label{e: f(y_{n+1})= y_{n}}
f(y_{n+1}^k)= y_n^k \text{ for every } n\in\N.
\end{equation}
We also define 
\begin{equation}
\alpha_{\max}\= \max\{\alpha_k: k\in N(f)\}.
\end{equation}
Since $E(f)$ is finite, by Property \ref{itm P3}, there is $\lambda>1$ such that
\begin{equation}\label{eq exp bound}
    Df(x) >\lambda \text{ \;for all\;}x\in I_k, k\in E(f).
\end{equation}
For the rest of this section, we denote by $\mathcal{P}$ the partition $\{I_1,\ldots,I_d\}$.
\subsection{Basic estimates}\label{subseq: Basic estimates}
For each $k\in N(f)$, let $\alpha_k$ and $b_k$ be given by Property \ref{itm P2}.
\begin{lemm}\label{Lemma M-P estimates}
    Let $k\in N(f)$. If $k\in\{1,\ldots,d-1\}$,
    \begin{equation}\label{eq x_n}
     x_n^k - \xi_k \sim (\alpha_kb_kn)^{-1/\alpha_k}, \;\; x_n^k - x_{n+1}^k \sim b_k^{-1/\alpha_k} (\alpha_k n )^{-(1+1/\alpha_k)} \text{\;\; as }n\to+\infty.
    \end{equation}
    and if $k\in \{2,\ldots,d\}$,
    \begin{equation}\label{eq y_n}
       \xi_k - y_{n}^k \sim (\alpha_kb_kn)^{-1/\alpha_k}, \;\;y_{n+1}^k - y_{n}^k \sim b_k^{-1/\alpha_k} (\alpha_k n )^{-(1+1/\alpha_k)} \text{\;\; as }n\to+\infty.
    \end{equation}
    Moreover, there is $C_0>0$ such that for every $n\in \mathbb{N}$, if $k\in N(f)\cap\{1,\ldots,d-1\}$ and $x\in(x^k_{n+1},x^k_n)$, or if $k\in N(f)\cap\{2,\ldots,d\}$ and $x\in(y^k_n,y^k_{n+1})$, we have
    \begin{equation}\label{eq Df^n}
        C_0^{-1}<Df^n(x)n^{-(1+1/\alpha_k)}<C_0.
    \end{equation}
\end{lemm}
\proof
    First assume that~$k\in\{1,\ldots,d-1\}$. Consider~$g:(I_k-\xi_k)\cap[0,+\infty)\to [0,+\infty)$ given 
    by~$g(x)= f(x+\xi_k) -\xi_k$. Let~$\widetilde{x}_n^k\= x_n^k - \xi_k$ for each~$n\in\N$. By Property \ref{itm P2}, the map $g$ satisfies $g(x)-x\sim b_k x^{1+\alpha_k}$ as $x\to 0^+$. Then by ~\cite[Sublemma 4.4]{CoatesMelbourneTalebi} we have
    \begin{equation*}
         \widetilde{x}_n^k \sim (\alpha_kb_kn)^{-1/\alpha_k}, \;\; \widetilde{x}_n^k - \widetilde{x}_{n+1}^k \sim b_k^{-1/\alpha_k} (\alpha_k n )^{-(1+1/\alpha_k)},
    \end{equation*}
    which proves~\eqref{eq x_n}. Similarly, for~$k\in\{2,\ldots,d\}$ we obtain~\eqref{eq y_n} by replacing $g$ 
    by 
    \begin{equation*}
    \widehat{g}(x) \= \xi_k-f(\xi_k-x).
    \end{equation*}
    We now prove~\eqref{eq Df^n}. We first treat the case $k\in N(f)\cap\{2,\ldots,d-1\}$; the endpoint cases are identical, with only one side of the neutral fixed point present. Since $\log Df|_{I_k}$ is Hölder-continuous of exponent $\min\{1,\alpha_k\}$, there is $D_k >0$ such that for every $x$ and $y$ in $I_k$ 
    \begin{equation}\label{e: logDf holder}
        |\log Df(x) - \log Df(y)|\leq D_k |x-y|^{\min\{1,\alpha_k\}}.
    \end{equation}
    On the other hand, from \eqref{eq x_n} we know that there exists $\widetilde{D}_k>0$ such that for every $n\in \N$
    \begin{equation}\label{e: x_n^k - x_{n+1}^k}
        \widetilde{D}_k^{-1}\leq (x_n^k-x_{n+1}^k) n^{1+1/\alpha_k}\leq \widetilde{D}_k. 
    \end{equation}
    Let 
    \begin{equation*}
        \widehat{D}_k\= D_k\widetilde{D}_k \sum_{j=1}^{+\infty} j^{-\min\{1,\alpha_k\}(1+1/\alpha_k)}.
    \end{equation*}
    Notice that $\widehat{D}_k\in(0,+\infty)$ as $\min\{1,\alpha_k\}(1+1/\alpha_k)>1$. By \eqref{e: logDf holder} and \eqref{e: x_n^k - x_{n+1}^k}, since $f^j((x_{n+1}^k,x_n^k))=(x_{n+1-j}^k,x_{n-j}^k)$ for every $j\in \{0,\ldots,n-1\}$, we obtain that for all $x$ and $y$ in $(x_{n+1}^k,x_n^k)$ we have
    \begin{equation*}
    \begin{split}
        \log\left|\frac{Df^n(x)}{Df^n(y)}\right| &\leq \sum_{j=0}^{n-1}|\log Df(f^j(x))-\log Df(f^{j}(y))| \leq D_k \sum_{j=0}^{n-1}|f^j((x_{n+1}^k,x_n^k))|^{\min\{1,\alpha_k\}}\\
        &\leq  D_k\sum_{j=0}^{n-1} |x_{n-j}^k-x_{n+1-j}^k|^{\min\{1,\alpha_k\}} \leq D_k\widetilde{D}_k \sum_{j=0}^{n-1} (n-j)^{-\min\{1,\alpha_k\}(1+1/\alpha_k)} \leq \widehat{D}_k.
        \end{split}
    \end{equation*}
    Thus, for every $x$ and $y$ in $(x_{n+1}^k,x_n^k)$ we have 
    \begin{equation}\label{distortion}
      \exp(-\widehat{D}_k) \leq \frac{Df^n(x)}{Df^n(y)}\leq \exp(\widehat{D}_k).
    \end{equation}
    Similarly, increasing $\widehat{D}_k$ if necessary, we obtain that \eqref{distortion} also holds for every $x$ and $y$ in $(y_n^k,y_{n+1}^k)$. Now, from \eqref{e: defi x_n^k} and \eqref{e: f(x_{n+1}) = x_{n}}, we get 
    $f^n((x_{n+1}^k,x_n^k))=\text{ int }\bigcup_{j=k+1}^{d}I_{j}$, and then by \eqref{eq x_n} and \eqref{distortion} there is $\widetilde{C}_k>0$ such that for every $x\in (x_{n+1}^k,x_n^k)$, 
    \begin{equation}\label{eq distortion 2}
        \widetilde{C}_k^{-1}<Df^n(x)n^{-(1+1/\alpha_k)}<\widetilde{C}_k.
    \end{equation}
    Again, increasing $\widetilde{C}_k$ if necessary, we get that \eqref{eq distortion 2} also holds for every $x\in(y_n^k,y_{n+1}^k)$. By the same argument, if $1\in N(f)$, we can find $\widetilde{C}_1>0$ such that for every $x\in (x_{n+1}^1,x_n^1)$
    \begin{equation*}
        \widetilde{C}_1^{-1}<Df^n(x)n^{-(1+1/\alpha_1)}<\widetilde{C}_1,
    \end{equation*}
    and if $d\in N(f)$, we can find $\widetilde{C}_d>0$ such that for every $x\in(y_{n}^d,y_{n+1}^d)$
    \begin{equation*}
        \widetilde{C}_d^{-1}<Df^n(x)n^{-(1+1/\alpha_d)}<\widetilde{C}_d
    \end{equation*}
    Therefore, taking $C_0\=\max_{k\in N(f)} \widetilde{C}_k$ we obtain \eqref{eq Df^n}, which concludes the proof of the lemma.
\endproof

\begin{rema}\label{rmk MP bound 1}
    It follows from \eqref{eq Df^n} that there is $\varepsilon_0>0$ such that for every $k\in N(f)$, every $n\in\N$, and every $x\in (x^k_{n+1},x^k_n)\cup(y^k_n,y^k_{n+1})$, we have
    \begin{equation}\label{MP bound}
    (1+\varepsilon_0 n)^{1+1/\alpha_k} \leq Df^n(x) \leq (1+\varepsilon_0^{-1} n)^{1+1/\alpha_k}.
\end{equation}
On the other hand, decreasing $\varepsilon_0>0$ if necessary, so that $(1+\varepsilon_0)^{^{1+1/\alpha_{\max}}}\leq \lambda$, we obtain that for every $k\in E(f)$, every $n\in \N$, and every $x\in h_k^n([0,1])$, we have
\begin{equation}
    Df^n(x) \geq (1+\varepsilon_0)^{n(1+1/\alpha_{\max})} \geq (1+\varepsilon_0 n)^{1+1/\alpha_{\max}}.
\end{equation}
\end{rema}

\begin{lemm}\label{lemma bound Df^n}
Let $\varepsilon_0>0$ be given by Remark \ref{rmk MP bound 1}. There exists a constant $C_1>1$ such that for every $n\in \N$ and every set of the form $Q\=h_{k_1}\circ\cdots\circ h_{k_{n+1}}([0,1])$ with either $k_{n+1}\neq k_{n}$ or $k_{n+1}\in E(f)$, we have
    \begin{equation}\label{distortion eq 1}
         \diam Q \leq (1+\varepsilon_0 n)^{-(1 + 1/\alpha_{\max})},
    \end{equation}
    and for all $x$ and $y$ in $\text{int }Q$
    \begin{equation}\label{distortion eq 2}
     Df^n(x) \geq (1+\varepsilon_0 n)^{1+1/\alpha_{\max}} \text{ and } C_1^{-1}\leq \frac{Df^n(x)}{Df^n(y)}\leq C_1.
    \end{equation}
    \end{lemm}
\begin{proof}
    We first consider a set $P$ of the form  $h_\ell^s\circ h_k([0,1])$ for some $s\in\N$, and $\ell,k \in \{1,\ldots,d\}$ with $\ell\neq k$ or $k\in E(f)$.  Observe that
    \begin{equation}\label{eq 1 rmk MP bound 2}
        f^s(P)= h_k([0,1])\subseteq \overline{I}_k \text{ and } f^j(P)= h^{s-j}_{\ell}\circ h_k([0,1])\subseteq \overline{I}_\ell \text{ for every }j=0,\ldots,s-1.
    \end{equation}
    If $\ell\in E(f)$, from Remark \ref{rmk MP bound 1} we have 
    \begin{equation}\label{eq bound for Df^s E(f)}
        Df^s(x) \geq (1+\varepsilon_0 s)^{1+1/\alpha_{\max}}\text{ for every } x\in \text{int } P.
    \end{equation}
    On the other hand, if $\ell\in N(f)$, we have that $\ell\neq k$, and then from \eqref{eq 1 rmk MP bound 2} we get
    \begin{equation*}
        P \subseteq h_{\ell}^s([0,1])\setminus h_{\ell}^{s+1}([0,1]).
    \end{equation*}
    It follows that
    \begin{equation}\label{eq lower bound Df^s}
        \text{int }P\subseteq 
        \begin{cases}
            (x_{s+1}^1, x_s^1)& \text{ if } \ell=1,\\
            (x_{s+1}^\ell, x_s^\ell)\cup(y_s^\ell, y_{s+1}^\ell) & \text{ if } \ell\in \{2,\ldots,d-1\},\\
            (y_{s+1}^d, y_s^d) & \text{ if } \ell=d.
        \end{cases}
    \end{equation}
    Thus, from Remark \ref{rmk MP bound 1} we obtain
    \begin{equation}\label{eq 2 rmk MP bound 2}
         Df^s(x) \geq (1+\varepsilon_0 s)^{1+1/\alpha_\ell}  \geq (1+\varepsilon_0 s)^{1+1/\alpha_{\max}} \text{ for every } x\in \text{int }P.
    \end{equation}

    Now, let $Q$ be a set as the one in the statement of the lemma.
    Let $m\in\N$, $n_1,\ldots,n_m\in \N$, and $\ell_1,\ldots,\ell_{m+1}\in \{1,\ldots,d\}$ be such that $\ell_j \neq \ell_{j+1}$ for every $j=1,\ldots,m-1$ and the word $\ell_{1}^{n_1} \ell_{2}^{n_2}\cdots \ell_{m}^{n_m}\ell_{m+1}$,  where $\ell_{j}^{n_j}$ denotes the concatenation of $\ell_j$ with itself $n_j$ times, satisfies
    \begin{equation}
        k_1k_2\cdots k_{n} k_{n+1} = \ell_{1}^{n_1} \ell_{2}^{n_2}\cdots \ell_{m}^{n_m}\ell_{m+1}.
    \end{equation}
    Then,
    \begin{equation}
        n_1+\ldots + n_m =n \text{\; and \;} \ell_{m+1}= k_{n+1}.
    \end{equation}
    It follows that $\ell_m\neq \ell_{m+1}$ or $\ell_{m+1}\in E(f)$. Put $n_0\=0$, by the chain rule we have 
    \begin{equation}\label{eq chain rule lemma bound Df^n}
        Df^n(x) = \prod_{j=1}^m Df^{n_j}(f^{n_0+\ldots + n_{j-1}}(x)).
    \end{equation}
    Notice that for each $j\in \{1,\ldots,m\}$,
    \begin{equation*}
        f^{n_0+\dots + n_{j-1}}(x) \in h_{\ell_j}^{n_j}\circ\cdots\circ h_{\ell_m}^{n_m} \circ h_{{\ell_{m+1}}}([0,1]) \subseteq h_{\ell_j}^{n_j}\circ h_{\ell_{j+1}}([0,1]).
    \end{equation*}
    Since for every $j\in \{1,\ldots,m-1\}$ we have $\ell_j \neq \ell_{j+1}$, and $\ell_m\neq \ell_{m+1}$ or $\ell_{m+1}\in E(f)$, from \eqref{eq bound for Df^s E(f)}, \eqref{eq 2 rmk MP bound 2}, and \eqref{eq chain rule lemma bound Df^n} we obtain that for each $x\in \text{int }Q$,
   \begin{equation*}
       \begin{split}
           Df^{n}(x) &\geq \prod_{j=1}^m (1+\varepsilon_0 n_j)^{(1+1/\alpha_{\max} )}\geq (1+\varepsilon_0(n_1+\ldots+n_m))^{1+1/\alpha_{\max}}= (1+\varepsilon_0 n)^{1+1/\alpha_{\max}}.
       \end{split}
   \end{equation*}
   Since $f^n$ maps $Q$ diffeomorphically onto $h_{k_{n+1}}([0,1])$, whose diameter is less than one, the lower bound on $Df^n$ gives
\begin{equation*}
\operatorname{diam} Q \le \left(\inf_{x\in\operatorname{int} Q} Df^n(x)\right)^{-1} \le (1+\varepsilon_0 n)^{-(1+1/\alpha_{\max})}.
\end{equation*}
 This proves \eqref{distortion eq 1} and the first assertion in \eqref{distortion eq 2}. For each $k\in E(f)$, choose $\alpha_k\in (0,1]$ so that $\log Df|_{I_k}$ is Hölder-continuous of exponent $\alpha_k$. Now, let $x$ and $y$ be in $Q$, then for a Hölder constant $D>0$ we have
    \begin{equation*}
    \begin{split}
        |\log Df^n(x) - \log Df^n(y)| &\leq \sum_{j=1}^{m} |\log Df^{n_j}(f^{n_0+\ldots+n_{j-1}}(x)) - \log Df^{n_j}(f^{n_0+\ldots+n_{j-1}}(y))|\\
        &\leq  \sum_{j=1}^m \sum_{k=0}^{n_j -1}|\log Df(f^{n_0+\ldots+n_{j-1} + k}(x)) - \log Df(f^{n_0+\ldots+n_{j-1} + k}(y))|\\
        &\leq D\sum_{j=1}^{m}\sum_{k=0}^{n_j-1} |f^{n_0+\ldots+n_{j-1} + k}(x) - f^{n_0+\ldots+n_{j-1} + k}(y)|^{\min\{1,\alpha_{\ell_j}\}}\\
        &\leq D\sum_{j=1}^{m}\sum_{k=0}^{n_j-1} |f^{n_0+\ldots+n_{j-1} + k}(Q)|^{\min\{1,\alpha_{\ell_j}\}}.
        \end{split}
    \end{equation*}
    If $\ell_j\in E(f)$, since $(1+\varepsilon_0)^{1+1/\alpha_{\max}}\leq \lambda$, by Remark \ref{rmk MP bound 1} we have that for every $k\in \{0,\ldots, n_j -1\}$,
    \begin{equation}
        |f^{n_0+\ldots+n_{j-1} + k}(Q)| \leq \frac{|f^{n_0+\ldots + n_j}(Q)|}{(1+\varepsilon_0 )^{(n_j - k)(1+1/\alpha_{\max})}},
    \end{equation}
    and if $\ell_j \in N(f)$, by \eqref{eq Df^n} and \eqref{eq lower bound Df^s}, for every $k\in\{0,\ldots,n_j-1\}$,
    \begin{equation}
        |f^{n_0+\ldots+n_{j-1} + k}(Q)| \leq \frac{|f^{n_0+\ldots + n_j}(Q)|}{(1+\varepsilon_0(n_j-k) )^{1+1/\alpha_{\ell_j}}}.
    \end{equation}
    Hence, taking $\widehat{C}>0$ greater than 
    \begin{equation*}
       \max_{\ell\in\{1,\ldots,d\}}D \sum_{k=0}^{+\infty} \left(\frac{1}{1+\varepsilon_0}\right)^{k\min\{1,\alpha_{\ell}\}(1+1/\alpha_{\max})} \text{\;and\;} \max_{\ell\in\{1,\ldots,d\}}D\sum_{k=0}^{+\infty} \left(\frac{1}{1+ \varepsilon_0 k}\right)^{\min\{1,\alpha_{\ell}\}(1+1/\alpha_{\ell})},
    \end{equation*}
    we obtain
    \begin{equation}\label{eq 2 MP bound}
        |\log Df^n(x) - \log Df^n(y)| \leq \widehat{C}\left( 1 + \sum_{j=1}^{m-1} |f^{n_1+\ldots + n_j}(Q)|^{\min\{1,\alpha_{\ell_j}\}}\right).
    \end{equation}
    Now, put 
    \begin{equation*}
        \widehat{\alpha}\=  (1+1/\alpha_{\max})\min_{i\in \{1,\dots,d\}}\min\{1,\alpha_i\},
    \end{equation*}
    and observe that for each $j\in\{1,\ldots,m-1\}$ we have
    \begin{equation*}
    \begin{split}
        &|f^{n_1+\ldots+n_j}(Q)|^{\min\{1,\alpha_{\ell_j}\}} \leq\\
        &\left(\prod_{k\in\{j+1,\ldots,m\}:\ell_k\in E(f)} \frac{1}{(1+\varepsilon_0)^{n_k(1+1/\alpha_{\max})}}\right)^{\min\{1,\alpha_{\ell_j}\}}\cdot\left(\prod_{k\in\{j+1,\ldots,m\}: \ell_k\in N(f)} \frac{1}{(1+\varepsilon_0 n_k)^{(1+1/\alpha_{\ell_k})}}\right)^{\min\{1,\alpha_{\ell_j}\}}\\
        &\leq  \prod_{k=j+1}^{m} \frac{1}{(1+\varepsilon_0)^{\widehat{\alpha}}}\\
       & = \left(\frac{1}{1+\varepsilon_0}\right)^{\widehat{\alpha}(m-j)}.
        \end{split}
    \end{equation*}
    Together with \eqref{eq 2 MP bound} this implies
    \begin{equation}
        |\log Df^n(x) - \log Df^n(y)| \leq \widehat{C}\sum_{j=1}^{m} \left(\frac{1}{1+\varepsilon_0}\right)^{\widehat{\alpha}(m-j)} = \widehat{C}\sum_{j=0}^{m-1} \left(\frac{1}{(1+\varepsilon_0)^{\widehat{\alpha}}}\right)^{j}. 
    \end{equation}
    Taking $C_1\= \exp(\widehat{C}\sum_{j=0}^{+\infty} \frac{1}{(1+\varepsilon_0)^{\widehat{\alpha}j}})$, we obtain the second part of \eqref{distortion eq 2}.
\end{proof}
\begin{lemm}\label{diam to 0}
    We have that
    \begin{equation}\label{diam to 0 eq0}
        \max_{Q\in\bigvee_{j=0}^{n}f^{-j}\mathcal{P}} \diam Q \to 0 \text{ as } n \to +\infty.
    \end{equation}
\end{lemm}
\begin{proof}
    Let $Q\in \bigvee_{j=0}^{n}f^{-j}\mathcal{P}$. Then there are $k_1,\ldots,k_{n+1} \in\{1,\ldots,d\}$ such that  
    \begin{equation}
    \overline{Q}=h_{k_1}\circ\cdots\circ h_{k_{n+1}}([0,1]).
    \end{equation}
    If there is an integer $j = \lceil n/2 \rceil,\dots, n$ such that $k_{j}\neq k_{j+1}$, by Lemma \ref{lemma bound Df^n} we have
    \begin{equation}\label{eq 2 diam to 0}
        \diam Q \leq \diam (h_{k_1}\circ \cdots \circ h_{k_{j+1}}([0,1])) \leq (1+\varepsilon_0 j)^{-(1+1/\alpha_{\max})} \leq \left(1+\frac{\varepsilon_0 n}{2}\right)^{-(1+1/\alpha_{\max})}.
    \end{equation}
    Now assume that there is $k\in\{1,\ldots,d\}$ with $k_j=k$ for every integer $j\in \{\lceil n/2 \rceil,\dots,n+1\}$. If $k\in E(f)$, we have
    \begin{equation}\label{eq 3 diam to 0}
        \diam Q \leq \diam f^{\left\lfloor \frac{n}{2}\right\rfloor}(Q)\leq \lambda^{-(n+1-\lfloor n/2\rfloor)}=\lambda^{-(\lceil \frac{n}{2}\rceil+1)}.
    \end{equation}
    If $k\in N(f)\cap\{2,\ldots,d-1\}$, then
    \begin{equation*}
        f^{\left\lfloor \frac{n}{2}\right\rfloor}(Q) \subseteq \left[y_{\lceil \frac{n}{2}\rceil+1}^k, x_{\lceil \frac{n}{2}\rceil+1}^k\right].
    \end{equation*}
    Thus,
    \begin{equation}\label{eq 4 diam to 0}
       \diam Q \leq \diam f^{\left\lfloor \frac{n}{2}\right\rfloor}(Q) \leq  x_{\lceil \frac{n}{2}\rceil+1}^k - y_{\lceil \frac{n}{2}\rceil+1}^k.
    \end{equation}
    If $k = 1$ and $1\in N(f)$, we get
    \begin{equation*}
        f^{\left\lfloor \frac{n}{2}\right\rfloor}(Q) \subseteq \left[0,x_{\lceil \frac{n}{2}\rceil+1}^1\right],
    \end{equation*}
    and it follows that
    \begin{equation}\label{eq 5 diam to 0}
        \diam Q \leq \diam f^{\left\lfloor \frac{n}{2}\right\rfloor}(Q) \leq x_{\lceil \frac{n}{2}\rceil+1}^1.
    \end{equation}
    Finally, if $k = d $ and $d\in N(f)$, we have
    \begin{equation*}
        f^{\left\lfloor \frac{n}{2}\right\rfloor}(Q) \subseteq \left[y_{\lceil \frac{n}{2}\rceil+1}^d,1\right],
    \end{equation*}
    and then 
    \begin{equation}\label{eq 6 diam to 0}
        \diam Q \leq \diam f^{\left\lfloor \frac{n}{2}\right\rfloor}(Q) \leq 1-y_{\lceil \frac{n}{2}\rceil+1}^d.
    \end{equation}
    Therefore, by  \eqref{eq 2 diam to 0}, \eqref{eq 3 diam to 0}, \eqref{eq 4 diam to 0}, \eqref{eq 5 diam to 0}, \eqref{eq 6 diam to 0}, and Lemma \ref{Lemma M-P estimates}, we obtain \eqref{diam to 0 eq0}.
    \end{proof}
    \begin{rema}\label{rema topologically exact}
        From Lemma \ref{diam to 0} we have that for every open set $U\subseteq [0,1]$, there are $N\in\N$ and $Q\in \bigvee_{{j=0}}^{N-1} f^{-j}\mathcal{P}$ contained in $U$. Since each branch is full, $f^N$ maps the interior of $Q$ onto $(0, 1)$. Hence $(0,1) \subseteq f^N(U)$. That is, the map $f$ is topologically exact on $(0,1)$.
    \end{rema}

\subsection{Continuous extension and topological pressure}\label{subseq: continuous extension}
The main goal of this subsection is to prove Lemma \ref{tree pressure}, which gives a topological characterization for the pressure of branch continuous potentials (see Definition \ref{branch Holder}). To deal with discontinuities both on $f$ and the potential, in Lemma \ref{extension lemma} we construct a continuous extension of $([0,1],f)$ that allows us to use known properties of continuous dynamical systems on compact metric spaces.
 Recall that $\text{int } I_k = (x_{k-1},x_k)$ for every $k\in\{1,\ldots,d\}$.
\begin{lemm}\label{extension lemma}
     Put $Y\= \bigcup_{n\in\N_0} f^{-n}(\{x_1,\ldots,x_{d-1}\})$. There is a totally ordered set $X$, endowed with the order topology, and a continuous surjective map $\pi: X\to[0,1]$ such that the following hold.
    \begin{enumerate}
        \item[1.] The sets $X\setminus \pi^{-1}(Y)$ and $[0,1]\setminus Y$ are equal as totally ordered sets, the map $\pi$ is the identity on $X\setminus \pi^{-1}(Y)$ and $\pi^{-1}(Y)$ consists of two disjoint copies $Y^{-}$ and $Y^{+}$ of $Y$ such that for every $y\in Y$ one has $\pi^{-1}(\{y\})=\{y^-,y^+\}$ and $y^-<y^+$.

        \item[2.] The order topology on $X$ is compact and metrizable.

        \item[3.] There is a continuous map $\widetilde{f}: X\to X$ such that $\pi\circ \widetilde{f} = f\circ\pi$,
        \begin{equation*}
            Y^- = \bigcup_{n\in\N_0}\widetilde{f}^{-n}(\{x_1^-,\ldots, x_{d-1}^-\}), \text{ and } Y^+= \bigcup_{n\in\N_0}\widetilde{f}^{-n}(\{x_1^+,\ldots, x_{d-1}^+\}).
        \end{equation*}

    \item[4.] The map $\widetilde{\pi}:X\to \{0,\ldots,d-1\}^{\N_0}$ given by 
\begin{equation*}
    (\widetilde{\pi}(x))_k= \begin{cases}
        0 & \text{ if }\widetilde{f}^k(x) \in [0,x_1^-],\\
        i & \text{ if }\widetilde{f}^k(x) \in [x_i^+,x_{i+1}^-],\\
        d-1 & \text{ if }\widetilde{f}^k(x)\in [x_{d-1}^+,1],
    \end{cases}
\end{equation*}
   is a topological conjugacy from $(X,\widetilde{f})$ to the full shift $(\{0,\ldots,d-1\}^{\N_0},\sigma)$.

    \item[5.] For every invariant measure $\nu$ for $f$ there is a unique invariant measure $\mu$ for $\widetilde{{f}}$ such that $\pi_*\mu=\nu$ and the systems $(X,\widetilde{f},\mu)$ and $([0,1],f,\nu)$ are isomorphic in measure. In particular, the map $\pi_*:\mathcal{M}_{\widetilde{f}}\to \mathcal{M}_f$ is one-to-one.
    \end{enumerate}
\end{lemm}
\begin{proof}
    The idea of the construction of the extension is well known. We follow \cite[Appendix A.5]{Keller2}. The set X is obtained from $[0,1]$ by doubling the points in $Y$. Thus, every $y\in Y$ is replaced by two points $y^-$ and $y^+$ and one declares that $y^-<y^+$. This endows $X$ with a total order. The order topology on $X$ is the topology generated by open order intervals and the intervals of the form $[0,b)$ and $(a,1]$ for all $a$ and $b$ in $X$. Denote by $\pi: X\to[0,1]$ the projection. Observe that $\pi$ is continuous since the preimage of every open interval in $[0,1]$ is an open order interval in $X$. This proves the first statement of the lemma and item 1.

    Observe that the total order of $X$ has the least upper bound property, so the order topology on $X$ is compact; see, for instance, \cite[Theorem 27.1]{munkres}. Also, the order topology is Hausdorff, which, together with the compactness, implies that $X$ is regular. Since it is second countable, by the Urysohn Metrization Theorem, see, for instance, \cite[Theorem 34.1]{munkres}, the order topology on $X$ is metrizable, proving item 2.

    The map $\widetilde{f}$ coincides with $f$ on $[0,1]\setminus Y$ and on $y^-$ and $y^+$ is defined by continuity. Thus, for every $k\in \{1,\ldots,d-1\}$ we have $\widetilde{f}(x_k^-)=1$ and $\widetilde{f}(x_k^+)=0$, and for every $y\in Y\setminus \{x_1,\ldots,x_{d-1}\}$ we have $\widetilde{f}(y^+) = f(y)^+$ and $\widetilde{f}(y^-)=f(y)^-$. Since $\widetilde{f}: [0,x_1^-]\to \widetilde{f}([0,x_1^-])$ and $\widetilde{f}: [x_{d-1}^+,1]\to \widetilde{f}([x_{d-1}^+,1])$ are increasing bijections, and $\widetilde{f}: [x_k^+,x_{k+1}^-]\to \widetilde{f}([x_k^+,x_{k+1}^-])$ is an increasing bijection for every $k\in\{1,\ldots, d-2\}$, then $\widetilde{f}$ is continuous. From the definition of $X$ and $\widetilde{f}$, we have $\pi\circ\widetilde{f}= f\circ\pi$, which finishes the proof of item 3.

    To show item 4, we first define a distance on $X$ compatible with the order topology. Define an atomic measure $\lambda$ by 
    \begin{equation}
        \lambda \= \sum_{n=0}^{+\infty} \sum_{y\in f^{-n}(\{x_1,\ldots,x_{d-1}\})} d^{-2n}\delta_y,
    \end{equation}
    and the increasing maps $\iota^-,\iota^+: [0,1]\to\R$ by
    \begin{equation}
        \iota^-(x)= x + \lambda([0,x)) \text{ and } \iota^+(x)= x + \lambda([0,x]).
    \end{equation}
    Observe that 
    \begin{equation*}
        \begin{split}
        &\iota^+(x) < \iota^-(x'), \text{ for } x<x',\\
        & \iota^+(x) = \iota^-(x), \text{ for } x\in[0,1]\setminus Y,\\
        & \iota^+(y) - \iota^{-}(y) = \lambda(\{y\}) \geq d^{-2n}, \text{ for } y\in f^{-n}(\{x_1,\ldots,x_{d-1}\}).
        \end{split}
    \end{equation*}
    We define $\iota:X\to\R$ by 
    \begin{equation}
        \iota(x) = \begin{cases}
            \iota^+(x), & \text{ if }x\in X\setminus Y^{-},\\
            \iota^-(x), & \text{ if }x\in Y^-.
        \end{cases}
    \end{equation}
    The set $\iota(X)$ is closed in $\R$, so its order topology coincides with the induced topology from $\R$. Furthermore, $\iota: X\to\iota(X)$ is an increasing bijection and thus a homeomorphism. For all $x$ and $x'$ in $X$ define the distance 
    \begin{equation}
        d(x,x')\=|\iota(x) - \iota(x')|. 
    \end{equation}
    Put $\widetilde{Y}\= \pi^{-1}(Y)$. Now, we prove that the map $\widetilde{\pi}$ in item 4 is a conjugacy. For every $n\in\N$, and every word $w$ on the alphabet $\{0,\ldots,d-1\}$ of length $n$ denote by $[w]$ the cylinder in $\{0,\ldots,d-1\}^{\N_0}$ that has the word $w$ in the first $n$ coordinates. We have that $\widetilde{\pi}^{-1}([w])$ is an interval $\widetilde{J}$ of the form $[a^+,b^-]$, $[0,b^-]$, or $[a^+,1]$ with $a$ and $b$ in $Y$. Since $[a^+,b^-] = (a^-,b^+)$, $[0,b^-]=[0,b^+)$ and $[a^+,1] = (a^-,1]$ we get that $\widetilde{\pi}$ is continuous. Notice that each of the intervals $(a,b]$, $(0,b]$ or $(a,1]$ is the image of $(0,1]$ by an inverse branch of $f^n|_{(0,1]}$, and denote by $J$ any of these intervals. Since there is no other point of $\pi^{-1}\left( \bigcup_{k=0}^{n-1} f^{-k}(\{x_1,\ldots,x_{d-1}\})\right)$ in $\widetilde{J}$ we have that
    \begin{equation}
        \text{diam}(\widetilde{J})\leq \text{diam}(J) + d^{-n}.
    \end{equation}
    By Lemma \ref{diam to 0} we have that diam$(J) \to 0$ as the length of $w$ goes to $+\infty$, then for every $\omega$ in $\{0,\ldots,d-1\}^{\N_0}$ there is a unique $x\in X$ such that $\widetilde{\pi}(x) = \omega$. Since $X$ is compact, we get that $\widetilde{\pi}$ is a homeomorphism, and since from the definition of $\widetilde{\pi}$ we have $\widetilde{\pi}\circ \widetilde{f}= \sigma\circ \widetilde{\pi}$, we conclude that the map $\widetilde{\pi}$ is a conjugacy.

    Observe that a subset $B$ of $X$ is the same as the set $\pi(B)$ in $[0,1]$ with the points in $\pi(B)\cap Y$ duplicated. Then the Borel $\sigma$-algebra of $X\setminus \widetilde{Y}$ is the same as the Borel $\sigma$-algebra of $[0,1]\setminus Y$. Since the set $Y$ does not support any invariant Borel probability measure for $f$, the same holds for $\widetilde{Y}$ and $\widetilde{f}.$ We deduce that for every invariant Borel probability measure $\mu$ for $\widetilde{f}$ we have that $\pi_*\mu$ is equal to the extension to the Borel $\sigma$-algebra of $[0,1]$ of the restriction of $\mu$ to the Borel $\sigma$-algebra of $X\setminus \widetilde{Y}$, which coincides with the Borel $\sigma$-algebra of $[0,1]\setminus Y$. On the other hand, for every invariant Borel probability measure $\nu$ for $f$, there is a unique $\mu$ in $\mathcal{M}_{\widetilde{f}}$ such that $\pi_*\mu = \nu$. Therefore, the dynamical systems $(X,\widetilde{f},\mu)$ and $([0,1],f,\nu)$ are measurably isomorphic. This finishes the proof of item 5 and the proof of the lemma.
\end{proof}

Let $(X,\widetilde{f})$ and $\pi$ be given by Lemma \ref{extension lemma}, and let $\varphi$ be branch continuous with respect to $f$. We can extend the potential $\varphi$ to $X$ as follows. Define $\widetilde{\varphi}:X\to\R$ by $\widetilde{\varphi}(x) \= \varphi(\pi(x))$ if $x\in X\setminus \pi^{-1}(Y)$, and for $y\in Y$,
    \begin{equation}\label{eq extended potential}
       \widetilde{\varphi}(y^{-}) \= \lim_{x\to y^{-}} \varphi(x),\hspace{3mm}  \widetilde{\varphi}(y^{+}) \= \lim_{x\to y^{+}} \varphi(x).
    \end{equation} 
    Note that the potential $\widetilde{\varphi}$ is continuous since the possible discontinuities of $\varphi$ are contained in $Y$.
    Since $Y$ is a zero measure set for every $f$-invariant probability measure, by part 5 of Lemma \ref{extension lemma} we have
    \begin{equation}\label{tree pressure eq1}
    \begin{split}
       P(\widetilde{\varphi},\widetilde{f}) &= \sup\left\{h_{\mu}(\widetilde{f}) + \int \widetilde{\varphi} \dd\mu : \mu \in \mathcal{M}_{\widetilde{f}}\right\}\\
       &= \sup\left\{h_{\nu}(f) + \int \varphi \dd\nu : \nu\in \mathcal{M}_f\right\} = P(\varphi).
    \end{split}
    \end{equation}
 Let $P_0 \= [0,x_1^-]$, $P_{d-1} \= [x_{d-1}^+,1]$, and for every $k\in\{1,\ldots,d-2\}$, let $P_k \= [x_k^{+},x_{k+1}^-]$. We define $\mathcal{C}\=\{P_0,\ldots,P_{d-1}\}$.

\begin{lemm}\label{tree pressure}
    Let $\varphi:[0,1]\to\R$ be branch continuous with respect to $f$. Then the following hold.
    \begin{enumerate}
    \item[1.] We have
    \begin{equation}\label{eq lemma tree pressure part 1}
        P(\varphi) = \lim_{n\to +\infty}\frac{1}{n} \log \sum_{P\in \bigvee_{k=0}^{n-1}f^{-k} \mathcal{P}} \sup_{x\in P}\exp(S_n\varphi(x)).
    \end{equation}
    \item[2.] For every $y\in(0,1)$,
    \begin{equation}\label{eq lemma tree pressure part 2}
        P(\varphi) = \lim_{n\to+\infty} \frac{1}{n} \log \sum_{x\in f^{-n}(y)} \exp(S_n \varphi(x)).
    \end{equation}
    \end{enumerate}
\end{lemm}
\begin{proof}
    Under the map $\widetilde{\pi}: X\to\{0,\dots,d-1\}^{\N_0}$ (see Lemma \ref{extension lemma}), the cover $\mathcal{C}=\{P_0,\ldots,P_{d-1}\}$ corresponds to the cover $\overline{\mathcal{C}}\=\{[0],\ldots, [d-1]\}$. 
Hence,
\begin{equation}\label{tree pressure eq2}
\begin{split}
    P(\widetilde{\varphi},\widetilde{f}) &= P(\widetilde{\varphi}\circ \widetilde{\pi}^{-1},\sigma ) = \lim_{n\to+\infty}\frac{1}{n}\log \sum_{P\in \bigvee_{j=0}^{n-1}\sigma^{-j}\overline{\mathcal{C}}} \sup_{x\in P}\exp\left(\sum_{k=0}^{n-1} \widetilde{\varphi}(\widetilde{\pi}^{-1}(\sigma^k(x)))\right) \\
    &=\lim_{n\to+\infty}\frac{1}{n}\log \sum_{P\in \bigvee_{j=0}^{n-1}\sigma^{-j}\overline{\mathcal{C}}} \sup_{x\in P}\exp\left(\sum_{k=0}^{n-1}\widetilde{\varphi}(\widetilde{f}^k(\widetilde{\pi}^{-1}(x)))\right)\\
    &= \lim_{n\to+\infty}\frac{1}{n}\log \sum_{P\in\bigvee_{j=0}^{n-1}\widetilde{f}^{-j}\mathcal{C}} \sup_{x\in P} \exp\left(\sum_{k=0}^{n-1} \widetilde{\varphi}(\widetilde{f}^k(x))\right).
    \end{split}
\end{equation}
From the definition of $\widetilde{\varphi}$, it is clear that 
\begin{equation*}
\sum_{P\in\bigvee_{j=0}^{n-1}\widetilde{f}^{-j}\mathcal{C}} \sup_{x\in P} \exp\left(\sum_{k=0}^{n-1} \widetilde{\varphi}(\widetilde{f}^k(x))\right) =  \sum_{P\in \bigvee_{k=0}^{n-1}f^{-k} \mathcal{P}} \sup_{x\in P}\exp(S_n\varphi(x)),
\end{equation*}
for every $n$ in $\N$. Therefore, from \eqref{tree pressure eq1} and \eqref{tree pressure eq2} we obtain \eqref{eq lemma tree pressure part 1}. This proves part 1.

For part 2, let $y\in(0,1)$ and $\varepsilon>0$. Since $\varphi$ is branch continuous, there exists $N_0 \in \mathbb{N}$ such that $\varphi$ is continuous on each atom of $\bigvee_{j=0}^{N_0-1} f^{-j}\mathcal{P}$ and admits continuous extensions to their boundaries. By Lemma \ref{diam to 0}  and uniform continuity, there exists $N\geq N_0$ such that, for every atom $P$ of $\bigvee_{j=0}^{N-1} f^{-j}\mathcal{P}$ and every $x_1,x_2\in P$, one has
\begin{equation*}
   |\varphi(x_1) - \varphi(x_2)|<\varepsilon. 
\end{equation*}
This implies that for every $n> N$, and every $x$ and $z$ in the same atom of $\bigvee_{j=0}^{n-1} f^{-j}\mathcal{P}$, 
\begin{equation*}
\begin{split}
    |S_n\varphi(x) - S_n\varphi(z)| &\leq \sum_{k=0}^{n-(N+1)} |\varphi(f^k(x)) - \varphi(f^k(z))| + \sum_{k=n-N}^{n-1} |\varphi(f^k(x)) - \varphi(f^k(z))|\\
    &\leq (n-N)\varepsilon + 2N\sup|\varphi|. 
    \end{split}
\end{equation*}
Thus, if $n> N$ and $P \in \bigvee_{k=0}^{n-1} f^{-k}\mathcal{P}$,
\begin{equation*}
    \sup_{x\in P}S_n\varphi(x) - S_n\varphi(z) \leq (n-N)\varepsilon + 2N\sup|\varphi|
\end{equation*}
for every $z\in P$. Since $f^{-n}(y)$ has exactly one element on each $P\in \bigvee_{k=0}^{n-1} f^{-k}\mathcal{P}$, for $n>N$ we have
\begin{equation*}
    \frac{1}{n} \log \sum_{P\in \bigvee_{k=0}^{n-1}f^{-k} \mathcal{P}} \sup_{x\in P}\exp(S_n\varphi(x)) \leq \frac{(n-N)\varepsilon + 2N\sup|\varphi|}{n} + \frac{1}{n} \log \sum_{x\in f^{-n}(y)}\exp(S_n\varphi(x)).
\end{equation*}
It follows that 
\begin{equation*}
    \lim_{n\to+\infty}  \frac{1}{n}\log \sum_{P\in \bigvee_{k=0}^{n-1}f^{-k} \mathcal{P}} \sup_{x\in P}\exp(S_n\varphi(x)) \leq \varepsilon + \liminf_{n\to+\infty} \frac{1}{n} \log \sum_{x\in f^{-n}(y)} \exp(S_n \varphi(x)).
\end{equation*}
Since $\varepsilon>0$ is arbitrary, we get
\begin{equation}\label{equation tree pressure lower bound}
     \lim_{n\to+\infty}  \frac{1}{n}\log \sum_{P\in \bigvee_{k=0}^{n-1}f^{-k} \mathcal{P}} \sup_{x\in P}\exp(S_n\varphi(x)) \leq \liminf_{n\to+\infty} \frac{1}{n} \log \sum_{x\in f^{-n}(y)} \exp(S_n \varphi(x)).
\end{equation}
Conversely, since each point of $f^{-n}(y)$ belongs to one atom of $\bigvee_{j=0}^{n-1} f^{-j}\mathcal{P}$, we have
\begin{equation}\label{equation tree pressure upper bound}
    \lim_{n\to+\infty}  \frac{1}{n}\log \sum_{P\in \bigvee_{k=0}^{n-1}f^{-k} \mathcal{P}} \sup_{x\in P}\exp(S_n\varphi(x)) \geq \limsup_{n\to+\infty} \frac{1}{n} \log \sum_{x\in f^{-n}(y)} \exp(S_n \varphi(x)).
\end{equation}
Combining \eqref{eq lemma tree pressure part 1}, \eqref{equation tree pressure lower bound}, and \eqref{equation tree pressure upper bound} we obtain \eqref{eq lemma tree pressure part 2}. This completes the proof of the lemma. 
\end{proof}
\begin{rema}\label{r: pressure via transfer op for the extension}
    An analogous argument to that given in the proof of part 2 of Lemma \ref{tree pressure} shows that for every continuous function $\psi: X \to \R$,
    \begin{equation*}
        P(\psi,\widetilde{f}) = \lim_{n\to+\infty}\frac{1}{n}\log \sum_{x\in \widetilde{f}^{-n}(y)} \exp(S_n\psi(x)) \text{ for every } y\in X.
    \end{equation*}
\end{rema}

\section{Proof of Theorem \ref{t: conformal measure}}\label{sec: proof of theo A} 

In this section we prove Theorem \ref{t: conformal measure}. We divide the proof into several lemmas.

\noindent Let $( X, \widetilde{f})$ be the continuous extension of $([0,1],f)$ given by Lemma \ref{extension lemma}, and let $\widetilde{\varphi}: X\to \R$ be the potential defined in \eqref{eq extended potential}. Also, recall the definition of the set ${Y}$ given by Lemma \ref{extension lemma}.

Given a continuous potential $\psi: X\to \R$, we define the \emph{transfer operator} $\widetilde{\mathcal{L}}_{\psi}$ acting on the set of bounded functions defined in $X$ as the bounded operator given by 
\begin{equation}\label{transfer operator}
   \widetilde{\mathcal{L}}_{\psi}h(y) \= \sum_{x\in \widetilde f^{-1}(y)} \exp(\psi(x)) h(x).
\end{equation}
By Lemma \ref{extension lemma}, $(X,\widetilde{f})$ is topologically conjugate to the full shift $(\{0,\dots,d-1\}^{\N_0},\sigma)$. Hence, $\widetilde{\mathcal{L}}_{\psi}h$ is continuous for each continuous function $h:X\to \R$. Moreover, by Remark \ref{r: pressure via transfer op for the extension},
\begin{equation}\label{e: pressure via transfer operator}
\begin{split}
    P(\psi,\widetilde{f}) &= \lim_{n\to+\infty}\frac{1}{n}\log\sum_{P\in \bigvee_{j=0}^{n-1}\widetilde{f}^{-j}\mathcal C} \sup_{y\in P} \exp(S_n\psi(y)) \\
    &= \lim_{n\to+\infty}\frac{1}{n}\log \widetilde{\mathcal{L}}_{\psi}1(x) \text{ for every } x\in X.
    \end{split}
\end{equation}
\begin{defi}\label{defi conformal measure 2}
   Let $M$ be a compact metric space and $T: M \to M$ be a Borel measurable map on $M$. Given a Borel measurable function $g: M\to [0,+\infty)$, a Borel probability measure $\nu$ on $M$ is $g$\emph{-conformal for} $T$ if for each Borel measurable set $A$ on which $T$ is injective we have
 \begin{equation*}
  \nu(T(A)) = \int_A g\dd\nu.
\end{equation*}
\end{defi}
\begin{lemm}\label{lemm existence conformal measure}
There exists an $\exp(P(\widetilde\varphi,\widetilde{f}) - \widetilde\varphi)$-conformal measure for $\widetilde{f}$.
\end{lemm}
\begin{proof}
By the Schauder--Tychonoff fixed point theorem (see \cite[p. 456]{DunfordSchwartz}), there exists a Borel probability measure $\widetilde m$ and a number $\lambda>0$ such that $\widetilde{\mathcal{L}}_{\widetilde\phi}^{*}\widetilde m=\lambda\widetilde m$. Then we have
     \begin{equation*}
         (\widetilde{\mathcal{L}}_{\widetilde{\varphi}}^*)^n \widetilde{m} = \lambda^n \widetilde{m}
     \end{equation*}
     for every $n\in\N$. It follows that
     \begin{equation}\label{eq1 conformal measure}
         \int \widetilde{\mathcal{L}}_{\widetilde{\varphi}}^n 1 \dd \widetilde{m} = \lambda^n.
     \end{equation}
     We get that for every $n\in \N$,
    \begin{equation*}
        \log \lambda =  \frac{1}{n}\log \int \widetilde{\mathcal{L}}_{\widetilde{\varphi}}^n 1 \dd\widetilde{m} \leq  \frac{1}{n}\log\sum_{P\in \bigvee_{j=0}^{n-1}\widetilde{f}^{-j}\mathcal C} \sup_{y\in P} \exp(S_n\widetilde{\varphi}(y)),
    \end{equation*}
   and by \eqref{e: pressure via transfer operator},
\begin{equation}\label{e: log lambda upper bound}
    \log\lambda \leq P(\widetilde{\varphi}, \widetilde{f}).
\end{equation}
    On the other hand, since $t\mapsto \log t$ is concave, from \eqref{eq1 conformal measure} and the dominated convergence theorem, we obtain
    \begin{equation*}
        \log \lambda \geq \lim_{n\to+\infty}\int \frac{1}{n} \log \widetilde{\mathcal{L}}_{\widetilde{\varphi}}^n 1 \dd \widetilde{m} = P(\widetilde{\varphi},\widetilde{f}).
    \end{equation*}
    Together with \eqref{e: log lambda upper bound} we conclude that $\lambda = \exp(P(\widetilde\varphi,\widetilde{f}))$, and then $\widetilde{\mathcal{L}}_{\widetilde\varphi}^* \widetilde{m} = \exp(P(\widetilde\varphi,\widetilde{f}))\widetilde{m}$. Using standard arguments, one can show that $\widetilde{m}$ is $\exp(P(\widetilde\varphi,\widetilde{f})-\widetilde\varphi)$-conformal for $\widetilde{f}$.
      \end{proof}

      \begin{lemm}\label{l: dense preimage tree}
          Let $\xi \in[0,1]$ be a fixed point of $f$. If $f^{-1}(\xi) \neq \{\xi\}$, then the preimage tree $\bigcup_{n=1}^{+\infty}f^{-n}(\xi)$ is a dense subset of $[0,1]$. In particular, if $\xi \not\in\{0,1\}$, then $\bigcup_{n=1}^{+\infty}f^{-n}(\xi)$ is dense in $[0,1]$.
      \end{lemm}
      \begin{proof}
        Assume that $f^{-1}(\xi) \neq \{\xi\}$. Then $f^{-1}(\xi)\cap(0,1)\neq \emptyset$. Indeed, if $f^{-1}(\xi) \subseteq \{0,1\}$, then $\xi \in\{0,1\}$ and $f(1-\xi)=\xi$, and this is a contradiction since $1-\xi$ is a fixed point.

        Let $x\in f^{-1}(\xi)\cap (0,1)$ and let $U\subseteq [0,1]$ be an open interval. By Remark \ref{rema topologically exact}, the map $f$ is topologically exact on (0,1), that is, there exists $N\in\N$ such that $(0,1) \subseteq f^{N}(U)$, and hence $f^{-N}(x)\cap U \neq \emptyset$. Since $x\in f^{-1}(\xi)$, we conclude that $\bigcup_{n=1}^{+\infty}f^{-n}(\xi)$ is dense in $[0,1]$. Finally, if $\xi \not\in \{0,1\}$, then $f^{-1}(\xi)\neq \{\xi\}$, and by the same argument we obtain that the preimage tree of $\xi$ is dense, concluding the proof of the lemma.
      \end{proof}

\begin{lemm}\label{lemma atomic conformal measure}
    Assume that $P(\varphi)= \varphi(\xi)$ for some fixed point $\xi$ with $f^{-1}(\xi) \neq \{\xi\}$, and such that
    \begin{equation}\label{eq series conformal measures}
        \sum_{n=1}^{+\infty}\sum_{x\in f^{-n}(\xi)\setminus\bigcup_{j=1}^{n-1}f^{-j}(\xi)} \exp(S_n\varphi(x) - nP(\varphi)) <+\infty.
    \end{equation}
    There exists an $\exp(P(\varphi)-\varphi)$-conformal measure $m$ with 
    \begin{equation}\label{eq1 atomic conformal measure}
        m\left(\bigcup_{n=1}^{+\infty}f^{-n}(\xi)\right) =1 \text{ and } m(\{x\})>0 \text{ for every }x \in \bigcup_{n=1}^{+\infty}f^{-n}(\xi).
    \end{equation}
    In particular, the measure $m$ is fully supported.
\end{lemm}
\begin{proof}
  We define a measure $\widehat{m}$ as follows. Put $\widehat{m}(\{\xi\})\=1$. For every $n\in \N$ and every $x\in f^{-n}(\xi)\setminus \bigcup_{j=1}^{n-1}f^{-j}(\xi)$, define
  \begin{equation}\label{eq defintion fully atomic conformal measure}
      \widehat{m}(\{x\})\= \exp(S_n\varphi(x) - nP(\varphi)).
  \end{equation}
  Finally, we define 
  \begin{equation*}
      \widehat{m}\left([0,1] \setminus \bigcup_{n=1}^{+\infty}f^{-n}(\xi)\right) \=0.
  \end{equation*}
  Then, we have
  \begin{equation*}
      \widehat{m}([0,1]) =1 +\sum_{n=1}^{+\infty}\sum_{x\in f^{-n}(\xi)\setminus\bigcup_{j=1}^{n-1}f^{-j}(\xi)} \exp(S_n\varphi(x) - nP(\varphi)). 
  \end{equation*}
  Define $m\= \widehat{m}/\widehat{m}([0,1])$. Since $P(\varphi)=\varphi(\xi)$, the measure $m$ is, by construction, an $\exp(P(\varphi)-\varphi)$-conformal measure satisfying \eqref{eq1 atomic conformal measure}. Finally, since $f^{-1}(\xi)\neq \{\xi\}$, by Lemma \ref{l: dense preimage tree} the preimage tree $\bigcup_{n=1}^{+\infty}f^{-n}(\xi)$ is dense in $[0,1]$, and therefore $m$ is fully supported.
\end{proof}

\begin{lemm}\label{lemma atoms of conformal measure}
    Let $m$ be an $\exp(P(\varphi)-\varphi)$-conformal measure for $f$. If $\varphi$ is hyperbolic for $f$, then $m$ is atom-free.
\end{lemm}
\begin{proof} Let $x\in [0,1]$. By conformality,
      \begin{equation}\label{eq2 conformal measure}
          m(\{f^{n}(x)\})= \exp(nP(\varphi) - S_{n} \varphi(x))m(\{x\}) \text{ for every }n\in\N.
      \end{equation}
   If $\varphi$ is hyperbolic for $f$, there exists $N\in \N$ such that
   \begin{equation}\label{e: conformality lemma 3.4}
       P(\varphi) - \sup \frac{1}{N} S_N \varphi >0.
   \end{equation}
      Notice that $S_{kN}\varphi(x) = S_N\varphi(x) + S_N\varphi(f^N(x))+\dots S_N \varphi(f^{(k-1)N}(x))$ for every $k\in\N$. Then, from \eqref{eq2 conformal measure}, we get
      \begin{equation}\label{eq 3 conformal measure}
          m(\{f^{kN}(x)\}) \geq \exp(k(NP(\varphi) - \sup S_N \varphi))m(\{x\}) \text{ for every }k\in\N.
      \end{equation}
    It follows that $m(\{x\})=0$; otherwise, by \eqref{e: conformality lemma 3.4} and \eqref{eq 3 conformal measure} we obtain \begin{equation*}
        m(\{f^{kN}(x)\})\xrightarrow[]{} +\infty \text{ as }k\to +\infty,
    \end{equation*}
      which contradicts the finiteness of $m$. Thus, the measure $m$ is atom-free provided that $\varphi$ is hyperbolic for $f$, completing the proof of the lemma.
\end{proof}
\begin{proof}[Proof of Theorem \ref{t: conformal measure}]
Let $\widetilde{m}$ be the $\exp(P(\widetilde{\varphi},\widetilde{f})-\widetilde{\varphi})$-conformal measure for $\widetilde{f}$ given by Lemma \ref{lemm existence conformal measure}. Let us first assume that $\widetilde{m}(\pi^{-1}(Y))=0$. In this case, the pushforward $m\= \pi_*\widetilde{m}$ is an $\exp(P(\varphi)-\varphi)$-conformal measure for $f$ with $m(Y)=0$, and then by conformality, we have $m((0,1))=1$. Let $U$ be an open subset of $[0,1]$, and let $N\in \N$ and $Q\in  \bigvee_{j=0}^{N-1}f^{-j}\mathcal{P}$ be as in Remark \ref{rema topologically exact}. Since $f^N$ is injective on $Q$ and $m$ is $\exp(P(\varphi)-\varphi)$-conformal for $f$,
\begin{equation*}
    1=m((0,1))= \int_{Q} \exp(NP(\varphi)-S_N\varphi) \dd m.
\end{equation*}
This implies that $m(Q)>0$ and since $Q$ is contained in $U$, we conclude that $m(U)>0$. That is, the measure $m$ has full support.

Now, assume that $\widetilde{m}(\pi^{-1}(Y))>0$. The conformality of $\widetilde{m}$ implies that $\widetilde{m}(\{0,1\})>0$. Fix $\xi\in \{0,1\}$ such that $\widetilde{m}({\{\xi\}})>0$. Using the conformality of $\widetilde{m}$,
\begin{equation}\label{e: finite series extension}
     \sum_{n=0}^{+\infty} \sum_{x\in \widetilde{f}^{-n}(\xi)\setminus \bigcup_{j=0}^{n-1}\widetilde{f}^{-j}(\xi)} \exp(S_n\widetilde{\varphi}(x) - nP(\widetilde{\varphi},\widetilde{f})) <+\infty.
\end{equation}
Since $\varphi$ is branch continuous with respect to $f$, there exists $N\in \N$ such that the restriction $\varphi|_{P}$ is continuous for every $P\in \bigvee_{j=0}^{N-1}f^{-j}\mathcal{P}$. Define 
\begin{equation}
    Y_N \= \bigcup_{j=0}^{N}f^{-j}(\{0,1\})\setminus\{0,1\}.
\end{equation}
Then,
\begin{equation}\label{e: tilde(varphi)=varphi}
    \widetilde{\varphi}(y^+) = \widetilde{\varphi}(y^-) = \varphi(y) \text{ for every }y\in Y\setminus Y_N.
\end{equation}
Let
\begin{equation*}
M\= \max_{y\in Y_N}\max\{|\varphi(y) - \widetilde{\varphi}(y^{-})|,|\varphi(y) - \widetilde{\varphi}(y^{+})|\},
\end{equation*}
which is a real number since $Y_N$ is finite.
By Lemma \ref{extension lemma} and \eqref{e: tilde(varphi)=varphi} we obtain that for every $n>N$ and every $x\in \widetilde{f}^{-n}(\xi)\setminus \bigcup_{j=0}^{n-1}\widetilde{f}^{-j}(\xi)$,
\begin{equation*}
    \begin{split}
        S_n\widetilde{\varphi}(x) &= S_{n-N}\widetilde{\varphi}(x) + S_N\widetilde{\varphi}(\widetilde{f}^{n-N}(x))\\
        & \geq S_{n-N}\varphi(\pi(x)) + S_{N}\varphi(\pi(\widetilde{f}^{n-N}(x))) - NM\\
        &= S_{n-N}\varphi(\pi(x)) + S_{N}\varphi(f^{n-N}(\pi(x))) - NM\\
        &= S_n\varphi(\pi(x)) - NM,
    \end{split}
\end{equation*}
and then, using \eqref{tree pressure eq1},
\begin{equation}\label{e: bound S_nvarphi by S_n tilde(varphi)}
    \exp(S_n\varphi(\pi(x)) - nP(\varphi)) \leq \exp(NM)\exp(S_n\widetilde{\varphi}(x)-nP(\widetilde{\varphi},\widetilde{f})).
\end{equation}
Since $f^{-n}(\{0,1\})\setminus\{0,1\} = \pi(\widetilde{f}^{-n}(\xi))\setminus\{\xi\}$ for each $n\in \N$,
combining \eqref{e: finite series extension} and \eqref{e: bound S_nvarphi by S_n tilde(varphi)} yields
\begin{equation*}
     \sum_{n=1}^{+\infty}\sum_{x\in f^{-n}(\xi)\setminus\bigcup_{j=1}^{n-1}f^{-j}(\xi)} \exp(S_n\varphi(x) - nP(\varphi)) <+\infty,
\end{equation*}
and
\begin{equation}\label{e: series 2 proof theorem A}
     \sum_{n=1}^{+\infty}\sum_{x\in f^{-n}(1-\xi)\setminus\bigcup_{j=1}^{n-1}f^{-j}(1-\xi)} \exp(S_n\varphi(x) - nP(\varphi)) <+\infty.
\end{equation}
Furthermore, since $\widetilde{m}$ is $\exp(P(\widetilde\varphi)-\widetilde{\varphi})$-conformal, $\widetilde{\varphi}(\xi)=\varphi(\xi)$, and $P(\widetilde\varphi) = P(\varphi)$, we have $P(\varphi)= \varphi(\xi)$.

If $f^{-1}(\xi)\neq \{\xi\}$, then by Lemma \ref{lemma atomic conformal measure}, there exists a fully supported $\exp(P(\varphi)- \varphi)$-conformal measure. On the other hand, if $f^{-1}(\xi)=\{\xi\}$, the measure $\delta_{\xi}$ is $\exp(P(\varphi)-\varphi)$-conformal, and moreover the convergence of the series in \eqref{e: series 2 proof theorem A} is equivalent to that of the series in \eqref{e: eq 1 thm A}. This proves the first assertion and part 3 of the theorem.

Part 1 follows from Lemma \ref{lemma atoms of conformal measure}. We are left to prove part 2.

Assume that $m$ is an $\exp(P(\varphi)-\varphi)$-conformal measure which is not fully supported. Then there exists an open set $U\subset [0,1]$ such that $m(U)=0$. Remark \ref{rema topologically exact}, together with the conformality of $m$, implies that $m((0,1))=0$, and hence $m(\{0,1\})=1$.

Let $\xi \in \{0,1\}$ be such that $m(\{\xi\})>0$. If $f^{-1}(\xi) \neq \{\xi\}$, by Lemma \ref{l: dense preimage tree} the preimage tree $\bigcup_{n=1}^{+\infty} f^{-n}(\xi)$ is dense in $[0,1]$, which implies that $m$ has full support, a contradiction. Hence $f^{-1}(\xi)=\{\xi\}$.

It follows that $f^{-1}(1-\xi) \neq \{1-\xi\}$, and therefore, again by Lemma \ref{l: dense preimage tree} we must have $m(\{1-\xi\})=0$. Consequently, $m=\delta_{\xi}$. By conformality, we have $P(\varphi)= \varphi(\xi)$, and by Lemma \ref{key lemma} we obtain that $\xi$ is a maximizing neutral fixed point. This completes part 2 and concludes the proof of Theorem \ref{t: conformal measure}.
\end{proof}
As an application of Theorem \ref{t: conformal measure} and \cite[Theorem 6]{Dobbs}, we obtain the following result, which will be used in the proof of Theorem \ref{Theo A}.
\begin{prop}\label{p: before proof thm C}
    Let $f \in \mathscr{F}$, and let $\varphi \in BH^\gamma(f)$ for some $\gamma \in (0,1]$. If $\varphi$ admits an equilibrium state with positive entropy, then 
      \begin{equation}\label{e: conv series section 5}
        \sum_{n=0}^{+\infty}\sum_{x\in f^{-n}(\{x_1,\dots,x_{d-1}\})}\exp(S_n\varphi(x)- nP(\varphi))=+\infty.
    \end{equation}
    In particular, there exists a fully supported $\exp(P(\varphi)-\varphi)$-conformal measure.
\end{prop}
\begin{proof}
    Let $\mu$ be an equilibrium state of $\varphi$ with positive entropy, and assume for the sake of contradiction that the series in \eqref{e: conv series section 5} converges. Since $\varphi \in BH^\gamma(f)$, there exists $N\in \N$ such that $\varphi|_P$ is Hölder-continuous for each $P\in \bigvee_{j=0}^{N-1}f^{-j}\mathcal{P}$. Consider the Borel measure $m$ on $\bigcup_{P\in \bigvee_{j=0}^{N-1}f^{-j}\mathcal{P}}\text{int }P$ given by
    \begin{equation}
      m(\{x\})\= \exp(S_n\varphi(x) - nP(\varphi)) \text{ for } x\in f^{-n}(\{x_1,\dots,x_{d-1}\}), n\geq N,
  \end{equation}
  and 
  \begin{equation*}
      m(B)=0 \text{ for every Borel set } B\subseteq \left( \bigcup_{P\in \bigvee_{j=0}^{N-1}f^{-j}\mathcal{P}}\text{int } P\right) \setminus \bigcup_{j=N}^{+\infty} f^{-j}(\{x_1,\dots,x_{d-1}\}).
  \end{equation*}
  The restriction of $f$ to $\bigcup_{P\in \bigvee_{j=0}^{N-1}f^{-j}\mathcal{P}}\text{int }P$ is a weak piecewise monotone cusp map in the sense of \cite[Definition 1]{Dobbs}, the measure $m$ is $(\exp(P(\varphi)-\varphi),0)$-conformal in the sense of \cite[Definition 2]{Dobbs}, and by Lemma \ref{l: dense preimage tree} it is fully supported. Thus, by \cite[Theorem 6]{Dobbs} we have $\mu\ll m$, and therefore the measure $\mu$ is atomic, which contradicts the fact that $h_{\mu}>0$. Hence, \eqref{e: conv series section 5} holds. Finally, by \eqref{e: conv series section 5} and Theorem \ref{t: conformal measure} there exists a fully supported $\exp(P(\varphi)- \varphi)$-conformal measure.
\end{proof}

\section{Keller spaces}\label{sec: Keller spaces}
In this section we recall some properties of the Banach spaces defined in \S\ref{subseq: Keller spaces}. Again, we follow the exposition of \cite{Juan1}, where a more detailed version of this section can be found.

Given $p\geq 1$ and a function $h\colon I\to\C$, put
\begin{equation*}
    \text{Var}_p(h)\= \sup\left\{\left(\sum_{i=1}^k |h(x_i)-h(x_{i-1})|^p\right)^{\frac{1}{p}} \colon k\geq 1,x_0,\ldots,x_k\in I, x_0<\ldots<x_k\right\},
\end{equation*}
and define
\begin{equation*}
    \lVert h \rVert_{BV_p} \= \text{Var}_p(h) + \lVert h \rVert_{\infty}. 
\end{equation*}
We say the function $h$ is of \emph{bounded $p$-variation} if $\lVert h \rVert_{BV_p}<+\infty$. Denote by $BV_p$ the space of all functions of bounded $p$-variation defined on $I$. Then $\lVert\cdot\rVert_{BV_p}$ is a norm on $BV_p$ and $(BV_p, \lVert\cdot\rVert_{BV_p} )$ is a Banach space. 

By \cite[Proposition 4.1]{Juan1}, we have that
\begin{equation}\label{contencion2 Keller}
    \lVert h \lVert_{\gamma,1}\leq 2^{\gamma}\lVert h \rVert_{BV_{1/\gamma}} \text{ for every } h\in BV_{1/\gamma},
\end{equation}
and there exists a constant $C_*>0$ such that
\begin{equation}\label{contencion3 Keller}
    \lVert h \rVert_{\infty} \leq C_* \lVert h \rVert_{\gamma,1} \text{ for every } h\in H^{\gamma,1}(m).
\end{equation}
Let $N\geq 2$ be an integer, and $\mathcal{P}\=\{I_1,\ldots,I_N\}$ be a partition of $I$ into intervals. Let $T\colon I\to I$ be a transformation on $I$ that is continuous and monotone on each $I_i\in \mathcal{P}$.

According to Definition \ref{branch Holder}, we say a function $h\colon I\to \R$ is \emph{branch Hölder-continuous of exponent $\gamma$ with respect to} $T$, if there is some $n$ in $\N$ such that for every $P \in \bigvee_{j=0}^{n-1}T^{-j}\mathcal{P}$, the restriction $h|_{P}$ is Hölder-continuous of exponent $\gamma$ and can be continuously extended to the boundary of $P$. We denote by $BH^{\gamma}(T)$ the set of all branch Hölder-continuous functions of exponent $\gamma$ with respect to $T$.

 Notice that every function in $BH^{\gamma}(T)$ is also of bounded $1/\gamma$-variation. Hence, together with \eqref{contencion2 Keller} and  \eqref{contencion3 Keller} we obtain
\begin{equation}\label{Holder is BV}
    BH^{\gamma}(T)\subset BV_{1/\gamma}\subset H^{\gamma,1}(m)\subset L^{\infty}(m).
\end{equation}
\begin{rema}
    In \eqref{Holder is BV} the second inclusion is an abuse of notation, as elements in $H^{\gamma,1}(m)$ are equivalence classes. Thus, the inclusion $BV_{1/\gamma}\subset H^{\gamma,1}(m)$ means that for every $h \in BV_{1/\gamma}$, the class of functions that are equal to $h$ almost everywhere with respect to $m$ is contained in $H^{\gamma,1}(m)$.
\end{rema}

Fix $p\geq 1$, let $g\colon I\to [0,+\infty)$ be a function of bounded $p$-variation, and let $\mathscr{L}_g$ be the operator defined by
\begin{equation}\label{transfer operator 2}
    \mathscr{L}_g(h)(x)\= \sum_{y\in T^{-1}(x)} h(y)g(y) = \sum_{i\in\{1,\ldots N\}, x\in T(I_i)
    } (h\cdot g)\circ (T_{|_{I_i}})^{-1}(x),
\end{equation}
for every bounded $h:I\to \C$.

Assume in addition that the measure $m$ satisfies
the following properties:
\begin{enumerate}[label=\textbf{H\arabic*.}, ref=H\arabic*]
    \item\label{itm H1} For each $I_i\in\mathcal{P}$, the map $T|_{I_i}^{-1}$ is non-singular with respect to $m$, so that for every subset $E$ of $I_i$ of measure zero, the set $(T|_{I_i}^{-1})^{-1}(E)= T(E)$ is also of measure zero;
    \item\label{itm H2} On a set of full measure with respect to $m$, we have
    \begin{equation*}
        g^{-1} = \sum_{i=1}^{N} \frac{\dd(T|^{-1}_{I_i})_{*} m}{\dd m};
    \end{equation*}
    \item\label{itm H3} For each bounded and measurable $h: I\to \C$ we have $\int_I \mathscr{L}_g(h)\dd m = \int_{I} h \dd m$, and $\mathscr{L}_g$ extends to a positive linear map from $L^{1}(m)$ to itself satisfying $\lVert\mathscr{L}_g(h)\rVert_1 \leq \lVert h\rVert_1$.
\end{enumerate}
The following theorem compiles results from \cite[Theorems 3.2 and 3.3]{Keller1}, and is established in \cite[Theorem 1]{Juan1}.
\begin{theo}\label{keller theo}
       Let $T, g, \mathscr{L}_g$, and $m$ be as above, and assume that there is an integer $n\geq 1$ such that the function
\begin{equation*}
    g_n(x) \= g(x)\cdots g(T^{n-1}(x))
\end{equation*}
satisfies $\sup_I g_n <1$. Then the following properties hold.
\begin{enumerate}
    \item[1.] The set $\mathscr{E}$ of eigenvalues of $\mathscr{L}_g|_{L^1(m)}$ of modulus 1 is finite. Moreover, for each $\lambda \in\mathscr{E}$, the space
    \begin{equation*}
        E(\lambda)\= \{h\in L^1(m) : \mathscr{L}_g(h) = \lambda h\}
    \end{equation*}
    is contained in $H^{1/p,1}(m)$ and it is of finite dimension.
    \item[2.] If for each $\lambda\in \mathscr{E}$ we denote by $\mathscr{P}(\lambda)$ the projection in $L^1(m)$ to $E(\lambda)$, then the operator 
    \begin{equation*}
        \mathscr{Q}\= \mathscr{L}_g - \sum_{\lambda\in \mathscr{E}}\lambda\mathscr{P}(\lambda)
    \end{equation*}
    satisfies $\sup\{\lVert \mathscr{Q}^n\rVert_1: n\geq0 \text{ integer}\}<+\infty$. Moreover, $\mathscr{Q}$ maps $H^{1/p,1}(m)$ to itself, and there is $\rho\in(0,1)$ and  a constant $M>0$ such that for every integer $n\geq0$ we have $\lVert \mathscr{Q}^n\lVert_{1/p,1}\leq M\rho^n$. Finally, for each $\lambda \in \mathscr{E}$ the operators $\mathscr{Q}\mathscr{P}(\lambda)$ and $\mathscr{P}(\lambda)\mathscr{Q}$ are both identically zero, and for each $\lambda'\in \mathscr{E}$ different from $\lambda$ the operators $\mathscr{P}(\lambda)\mathscr{P}(\lambda')$ and $\mathscr{P}(\lambda')\mathscr{P}(\lambda)$ are also identically zero.
    \item[3.] The set $\mathscr{E}$ contains $1$, and if we put $h\= \mathscr{P}(1)(\boldsymbol{1})$, then $\nu\= hm$ is a probability measure that is invariant by $T$ and that is an equilibrium state of $T$ for the potential $\log g$. Moreover, if $\eta$ is a $T$-invariant probability measure such that $\eta \ll m$, then $\eta \ll \nu$.
\end{enumerate}
\end{theo}
The following corollary can be found in \cite[Corollary 4.4]{Juan1}.
\begin{coro}\label{corollary keller theo}
    Under the assumptions of Theorem \ref{keller theo}, and  assuming in addition that $T$ is topologically exact on $I$, we have the following properties:
\begin{enumerate}
    \item[1.] The number $1$ is an eigenvalue of $\mathscr{L}_g$ of algebraic multiplicity 1. Moreover, there is $\rho \in (0,1)$ such that the spectrum of $\mathscr{L}_g|_{H^{1/p,1}(m)}$ is contained in $B(0,\rho)\cup \{1\}$.
    \item[2.] There is a constant $C>0$ such that for every bounded measurable function $\phi: I\to\C$, and every function $\psi \in H^{1/p ,1}(m)$, the measure $\nu$ given by part $3$ of Theorem \ref{keller theo} satisfies for every integer $n\geq 1$ that
    \begin{equation*}
       \left| \int \phi\circ T^n \cdot \psi \dd\nu - \int \phi \dd\nu\int \psi \dd\nu \right|\leq C\lVert \phi \rVert_{\infty}\lVert \psi\rVert_{1/p,1}\rho^n.
    \end{equation*}
    \item[3.] Given $\psi \in H^{1/p,1}(m)$, for each $\tau\in \C $ the operator $\mathscr{L}_{\tau}$ defined by
    \begin{equation*}
        \mathscr{L}_{\tau}(h) \= \mathscr{L}_g(\exp(\tau\psi)\cdot h)
    \end{equation*}
    maps $H^{1/p,1}(m)$ to itself and the restriction $\mathscr{L}_{\tau}|_{H^{1/p,1}(m)}$ is bounded. Moreover, $\tau\mapsto \mathscr{L}_{\tau}|_{H^{1/p,1}(m)}$ is analytic in the sense of Kato on $\C$, and the spectral radius of $\mathscr{L}_{\tau}|_{H^{1/p,1}(m)}$ depends in a real analytic way on $\tau$ on a neighborhood of $\tau=0$.
\end{enumerate}
\end{coro}

\begin{rema}
    In Corollary \ref{corollary keller theo} the map $T$ is required to be topologically exact on $I$. However, the proof in \cite{Juan1} remains valid under the weaker assumption that $T$ is topologically exact on a subset of 
    $[0,1]$ with full measure for $m$. Since $m$ has no atoms, we have $m(0,1)=1$, and hence Corollary \ref{corollary keller theo} applies to the maps in our family $\mathscr{F}$, as they are topologically exact on $(0,1)$ (see Remark \ref{rema topologically exact}).
\end{rema}

\section{Proofs of Theorems \ref{general theo A} and \ref{Theo A}}\label{sec: Proof of Theo B and C}
We begin this section by proving Theorem \ref{general theo A}. Having proven all the results in \Cref{sec: geometric estimates and top pressure,sec: proof of theo A}, its proof is essentially the same as that of \cite[Theorem B]{Juan1}, except for minor modifications. Then, we deduce Theorem \ref{Theo A} as a consequence of Theorem \ref{general theo A} and Lemma \ref{key lemma}.

\proof[Proof of Theorem \ref{general theo A}]
Since $\varphi$ is hyperbolic for $f$, there is $N\in \N$ such that $\sup \frac{1}{N}S_N\varphi<P(\varphi)$. 

Define the operator $\widehat{\mathcal{L}}_\varphi \= \exp(-P(\varphi))\mathcal{L}_\varphi$, where $\mathcal{L}_{\varphi}$ is the operator defined in \eqref{eq transfer operator 3}. Note that if we put $g\=\exp(\varphi- P(\varphi))$, then $\widehat{\mathcal{L}}_\varphi$ coincides with the operator $\mathscr{L}_{g}$ defined in \eqref{transfer operator 2} with $T$ replaced by $f$, and $I$ replaced by $[0,1]$.

By \eqref{Holder is BV} the function $\varphi$ is of bounded $(1/\gamma)$-variation. Since $\varphi$ is bounded, the constant $C\=\sup \exp(\varphi)$ is finite. So for all $x_1$ and $x_2$ in $[0,1]$, we have
\begin{equation*}
    |\exp(\varphi(x_1)) -\exp(\varphi(x_2))|\leq C|\varphi(x_1)-\varphi(x_2)|.
\end{equation*}
This implies that $\exp(\varphi)$, and so $g$, is of bounded $(1/\gamma)$-variation. On the other hand, since $\sup \frac{1}{N}S_N\varphi < P(\varphi)$, the function $g_N(x):=g(x)\cdots g(f^{N-1}(x))$ satisfies
\begin{equation*}
    \sup_{x\in [0,1]} g_N(x) = \sup_{x\in [0,1]}\exp(S_N\varphi(x) - NP(\varphi))<1.
\end{equation*}
Using that $m$ is a $(g^{-1})$-conformal measure for $f$, replacing $T$ by $f$ and $I$ by $[0,1]$, properties \ref{itm H1} and \ref{itm H2} in \Cref{sec: Keller spaces} hold with $p=1/\gamma$. Moreover, since $m$ is atom-free, then $m(Y)=0$, where $Y\subset[0,1]$ is the set introduced in Lemma \ref{extension lemma}. Then $m$ can be lifted to a measure $\widetilde{m}\=(\pi|_{[0,1]\setminus Y})^{-1}_*m$ that is $\exp(P(\widetilde\varphi,\widetilde{f}) - \widetilde{\varphi})$-conformal for $\widetilde{f}$, with $\widetilde{\varphi}$ as in \eqref{eq extended potential}. This implies that $\widetilde{\mathcal{L}}_{\widetilde{\varphi}}^*\widetilde{m} = \exp(P(\widetilde{\varphi},\widetilde{f}))\widetilde{m}$, where $\widetilde{\mathcal{L}}_{\widetilde{\varphi}}$ is the transfer operator defined in \eqref{transfer operator}, and then Property \ref{itm H3} in \Cref{sec: Keller spaces} holds. Therefore, for each $\widehat{\gamma} \in (0,\gamma]$, Theorem \ref{keller theo} and Corollary \ref{corollary keller theo} hold with $T$ and $I$ replaced by $f$ and $[0,1]$ respectively. Let $A$ be the constant given by Theorem \ref{keller theo} and $H^{\widehat{\gamma},1}(m)$ be the corresponding Banach space.

By the considerations above, we only need to prove the uniqueness of the equilibrium state and part 3 of the theorem. Let $\nu$ be an ergodic equilibrium state for $\varphi$. It follows that
\begin{equation*}
    h_{\nu} = P(\varphi) -\int \varphi \dd\nu = P(\varphi) - \int \frac{1}{N}S_N\varphi \dd \nu \geq P(\varphi) - \sup\frac{1}{N}S_N\varphi>0.
\end{equation*}
By Ruelle's inequality (see \cite[Theorem 2]{Hofbauer91} and \cite[Theorem 7.1]{BarrioJimenez} for a version that applies to our setting) we have
\begin{equation*}
    \chi_{\nu}\geq h_{\nu}>0.
\end{equation*}
Then by \cite[Theorem 6]{Dobbs} we have $\nu\ll m$ and, using part 3 of Theorem \ref{keller theo}, we conclude the uniqueness of the equilibrium state.

Let $\chi$ be a branch Hölder-continuous function of exponent $\overline{\gamma}\in (0,1]$. It remains to show that the function $t\mapsto P(\varphi + t\chi)$ is real analytic on a neighborhood of $t=0$. For each $\tau\in\C$, let $\mathscr{L}_{\tau}$ be the operator defined in part 3 of Corollary \ref{corollary keller theo} with $T$ replaced by $f$, $\psi$ replaced by $\chi$, and $I$ replaced by $[0,1]$. On the other hand, for each $t\in \R$ put
\begin{equation*}
    \varphi_t \= \varphi + t\chi\text{ and } g_t\= \exp(\varphi_t - P(\varphi_t)),
\end{equation*}
and note that the operator $\exp(P(\varphi)-P(\varphi_t))\mathscr{L}_{t}$ coincides with the operator $\mathscr{L}_{g}$ defined in \eqref{transfer operator 2} with $g$ replaced by $g_t$, $T$ replaced by $f$, and $I$ replaced by $[0,1]$. Let $\widehat{\gamma}\in(0,1]$ be sufficiently small so that both $\varphi$ and $\chi$ are branch Hölder-continuous of exponent $\widehat{\gamma}$, let $p_1$ and $p_2$ be such that
\begin{equation*}
    \sup\frac{1}{N}S_N\varphi < p_1 < p_2 < P(\varphi),
\end{equation*}
and let $\varepsilon_0$ be small enough so that
\begin{equation*}
    \varepsilon_0 \sup|\chi| < \min\{p_2-p_1, P(\varphi)-p_2\}.
\end{equation*}
Note that by our choice of $\varepsilon_0$, for every $t\in (-\varepsilon_0,\varepsilon_0)$ we have
\begin{equation*}
    \varphi_t > \varphi - (P(\varphi)-p_2),
\end{equation*}
so
\begin{equation*}
    \begin{split}
        P(\varphi_t)&\geq
         P(\varphi-(P(\varphi)-p_2))=p_2.
    \end{split}
\end{equation*}
On the other hand,
\begin{equation*}
    \sup \frac{1}{N}S_N \varphi_t \leq \sup \frac{1}{N}S_N \varphi + \varepsilon_0\sup |\chi| < p_1 + (p_2-p_1) = p_2. 
\end{equation*}
Therefore,
\begin{equation*}
    \sup\frac{1}{N}S_N\varphi_t < P(\varphi_t).
\end{equation*}
That is, for every $t\in (-\varepsilon_0,\varepsilon_0)$ the potential $\varphi_t$ is hyperbolic for $f$. It follows that the function $(g_t)_N(x) \= g_t(x)\cdots g_t(f^{N-1}(x))$ satisfies 
\begin{equation*}\sup (g_t)_N = \sup \exp(S_N \varphi_t(x) - NP(\varphi_t))<1,
\end{equation*}
and by Theorem \ref{t: conformal measure} there is an atom-free $\exp(P(\varphi_t)-\varphi_t)$-conformal measure $m_t$ for $f$. Since $g_t$ is of bounded $(1/\widehat{\gamma})$-variation, as before it follows that properties \ref{itm H1}, \ref{itm H2} and \ref{itm H3} in \Cref{sec: Keller spaces} hold with $p=1/\widehat{\gamma}$, and with $g,T,I$ and $m$ replaced by $g_t,f,I$ and $m_t$ respectively. We can thus apply part 1 of Corollary \ref{corollary keller theo} to conclude that $\exp(P(\varphi_t)-P(\varphi))$ is equal to the spectral radius of $\mathscr{L}_t$. Moreover, by part 3 of the same corollary, the function $t\mapsto \exp(P(\varphi_t)-P(\varphi))$ is real analytic on a neighborhood of $t=0$. This completes the proof of Theorem \ref{general theo A}.
\endproof

In the following proposition, we prove that the hypothesis in part 1 of Theorem \ref{Theo A} is equivalent to the hyperbolicity of the potential. Similar characterizations were proved in \cite{InoquioRivera} for continuous maps.

\begin{prop}\label{prop characterization of hyperbolicity}
    Let $f\in \mathscr{F}$, and let $\varphi \in BH^{\gamma}(f)$ for some $\gamma \in (0,1]$. Then $\varphi$ is hyperbolic for $f$ if and only if $P(\varphi)>\max_{k\in N(f)}\varphi(\xi_k)$.
\end{prop}
\begin{proof}
    Let $(X,\widetilde{f})$ be the continuous extension given by Lemma \ref{extension lemma} and let $\widetilde{\varphi}$ be as in \eqref{eq extended potential}. It is well known (see, for instance, \cite[Proposition 3.1]{InoquioRivera}, and \cite[Theorem A.3]{Morris1}) that
    \begin{equation*}
        \inf_{n\in \N} \sup_{x\in X} \frac{1}{n}S_n\widetilde{\varphi}(x) = \sup_{\mu\in \mathcal{M}_{\widetilde{f}}} \int \widetilde{\varphi}\dd\mu.
    \end{equation*}
    It follows from Lemma \ref{extension lemma} and \eqref{eq extended potential} that
    \begin{equation}\label{eq 1 characterization of hyperbolicity}
        \inf_{n\in \N} \sup_{x\in [0,1]}\frac{1}{n} S_n \varphi(x) = \sup_{\mu\in \mathcal{M}_f}\int \varphi \dd\mu.
    \end{equation}
    Now, assume that $\varphi$ is hyperbolic for $f$, then there exists $N\in\N$ such that
    \begin{equation*}\sup_{x\in [0,1]}\frac{1}{N}S_N\varphi(x) < P(\varphi).\end{equation*}
    From \eqref{eq 1 characterization of hyperbolicity} we have
    \begin{equation*}
         \sup_{\mu \in \mathcal{M}_f}\int \varphi \dd\mu<P(\varphi),
    \end{equation*}
    and in particular, $P(\varphi)> \max_{k\in N(f)}\varphi(\xi_k)$, proving the first implication.
    Now assume that $P(\varphi)>\max_{k\in N(f)} \varphi(\xi_k)$. Together with Lemma \ref{key lemma}, we obtain that $P(\varphi)>\sup_{\mu\in \mathcal{M}_f}\int \varphi \dd\mu$. Thus, by \eqref{eq 1 characterization of hyperbolicity} there is $N\in \N$ such that
    \begin{equation*}
         \sup_{x\in[0,1]}\frac{1}{N}S_N \varphi(x)<P(\varphi).
    \end{equation*}
    That is, $\varphi$ is hyperbolic for $f$. This completes the proof.
\end{proof}
\begin{proof}[Proof of Theorem \ref{Theo A}]
    Assume that part 1 does not hold. From \eqref{pressure} we know that
    \begin{equation}
    P(\beta\varphi) \geq \beta\max_{k\in N(f)}\varphi(\xi_k) \text{ for every } \beta>0.
    \end{equation}
     Then, by Proposition \ref{prop characterization of hyperbolicity}, there is $\beta_0>0$ such that $P(\beta_0 \varphi)= \beta_0 \max_{k\in N(f)}\varphi(\xi_k)$. Let $\beta_*$ be as in \eqref{eq 1 thm C}. Part (a) follows directly from Proposition \ref{prop characterization of hyperbolicity}. 

     Now, let $\beta\geq \beta_*$. Notice that
    \begin{equation*}
        \frac{1}{\beta}P(\beta\varphi) = \sup_{\mu\in \mathcal{M}_f} \left\{\frac{1}{\beta}h_\mu + \int \varphi \dd\mu\right\}\leq \sup_{\mu \in \mathcal{M}_f}\left\{\frac{1}{\beta_*}h_\mu + \int \varphi \dd\mu\right\} = \frac{1}{\beta_*}P(\beta_* \varphi) = \max_{k\in N(f)}\varphi(\xi_k),
    \end{equation*}
    and so \eqref{eq 2 thm 1} follows. For each $k\in N_{\max}(f)$, the measure $\delta_{\xi_k}$ is an ergodic equilibrium state of $\beta\varphi$. On the other hand, by Lemma \ref{key lemma}, any other ergodic equilibrium state must have positive entropy. Assume that $\mu$ is an ergodic equilibrium state for $\beta_*\varphi$ with positive entropy, and let $m_{\beta_*}$ be the fully supported $\exp(P(\beta_*\varphi)- \beta_*\varphi)$-conformal measure for $f$ given by Proposition \ref{p: before proof thm C}. By Ruelle's inequality (see \cite[Theorem 2]{Hofbauer91} and \cite[Theorem 7.1]{BarrioJimenez}) we get that
    \begin{equation*}
        \chi_\mu \geq h_\mu>0.
    \end{equation*}
    By \cite[Theorem 6]{Dobbs}, we obtain that $\mu\ll m_{\beta_*}$ and the density $h$ is bounded from below by a constant $c>0$ on an open interval $W\subseteq [0,1]$. Since $m_{\beta_*}$ is $\exp(P(\beta_*\varphi)-\beta_*\varphi)$-conformal, for every continuous function $\psi:[0,1]\to \R$ we have
    \begin{equation*}
    \begin{split}
        \int \psi\mathcal{L}_{\beta_*\varphi} h \dd m_{\beta_*} &= \int \mathcal{L}_{\beta_*\varphi}(h(\psi\circ f)) \dd m_{\beta_*}= e^{P(\beta_*\varphi)}\int h(\psi\circ f)\dd m_{\beta_*} = e^{P(\beta_*\varphi)}\int \psi \circ f \dd \mu\\
        &= e^{P(\beta_*\varphi)}\int \psi \dd \mu = e^{P(\beta_*\varphi)}\int \psi h\dd m_{\beta_*}.
        \end{split}
    \end{equation*}
    It follows that $\mathcal{L}_{\beta_*\varphi}h=e^{P(\beta_*\varphi)}h$. Now, let $x_0\in(0,1)$. Since $f$ is topologically exact on $(0,1)$ (see Remark \ref{rema topologically exact}), there exists $N\in \N$ such that $(0,1)\subseteq f^N(W)$. In particular, there is $y_0\in W\cap f^{-{N}}(x_0)$, and therefore
    \begin{equation*}
        h(x_0)= e^{-NP(\beta_*\varphi)}\sum_{y\in f^{-N}(x_0)} e^{\beta_*S_N\varphi(y)}h(y) \geq e^{\beta_*\inf (S_N \varphi-NP(\beta_*
        \varphi)}c>0.
    \end{equation*}
    Thus, the density $h$ is bounded from below by a positive constant. This implies that $m_{\beta_*}\ll \mu$ and then $\mu$ is fully supported. Moreover, if $\nu$ is another ergodic equilibrium state with positive entropy the argument above implies that $\nu\ll m_{\beta_*}$ and $m_{\beta_*}\ll \nu$. Thus, the measures $\mu$ and $\nu$ have the same sets of measure zero, which is impossible, since distinct ergodic measures are mutually singular. Hence, the potential $\beta_*\varphi$ admits at most one equilibrium state with positive entropy, and therefore, there exist at most $|N_{\max}(f)| + 1$ ergodic equilibrium states for $\beta_*$. Now, let $\beta>\beta_*$, and assume by contradiction that $\mu$ is an ergodic equilibrium state with positive entropy for $\beta\varphi$. It follows from \eqref{eq 2 thm 1} that
    \begin{equation*}
        \frac{1}{\beta_*}P(\beta_* \varphi) = \frac{1}{\beta}P(\beta\varphi) = \frac{1}{\beta}h_{\mu} + \int \varphi \dd\mu < \frac{1}{\beta_*}h_\mu + \int \varphi \dd\mu.
    \end{equation*}
    Hence,
    \begin{equation*}
        P(\beta_*\varphi) < h_\mu + \beta_* \int\varphi \dd\mu,
    \end{equation*}
    a contradiction.
    Finally, part (c) is a direct consequence of part 3 of Theorem \ref{general theo A} together with \eqref{eq 2 thm 1}.
\end{proof}
We finish this section with the following remark regarding the existence of equilibrium states with positive entropy for non-hyperbolic potentials.

\begin{rema}
    If $\varphi \in BH^{\gamma}(f)$ is not hyperbolic for $f$ and has an ergodic equilibrium state $\mu$ with positive entropy, then necessarily the series in \eqref{eq series conformal measures} diverges for every maximizing neutral fixed point $\xi$ with $f^{-1}(\xi)\neq \{\xi\}$. Otherwise, by Lemma \ref{lemma atomic conformal measure}, there would exist a fully supported $\exp(P(\varphi)-\varphi)$-conformal measure $m$ that gives full measure to $\bigcup_{n=0}^{+\infty}f^{-n}(\xi)$ for some maximizing neutral fixed point $\xi$ as above, and by \cite[Theorem 6]{Dobbs} the measure $\mu$ would be atomic.
\end{rema}

\appendix
\section{Key lemma.}\label{sec: key lemma}
In this appendix we prove Lemma \ref{key lemma}, which we refer to as the \textbf{Key Lemma} and is the main result needed to deduce Theorem \ref{Theo A} from Theorem \ref{general theo A}. Let $f$ be a map in $\sF$, and consider the collection of $f$-invariant probability measures
\begin{equation}
\mathcal{S}\=\left\{\sum_{k\in N(f)} t_k\delta_{\xi_k} : t_k\geq0 \text{ for every } k\in N(f) \text{ and } \sum_{k\in N(f)}t_k=1\right\}.
\end{equation}
\begin{lemm}[Key Lemma]\label{key lemma}
Let $\gamma\in(0,1]$ and let $\varphi:[0,1]\to\R$ be branch Hölder-continuous of exponent $\gamma$. Then for every $\mu\in \mathcal{M}_f \setminus\mathcal{S}$ we have
\begin{equation}\label{eq key lemma}
    P(\varphi)>\int \varphi \dd\mu.
\end{equation}
\end{lemm}
Let $\mu \in \mathcal{M}_f\setminus \mathcal{S}$. By the ergodic decomposition theorem, we can assume that $\mu$ is ergodic. Since $\mu$ is not in $\mathcal{S}$, either $\mu(I_k\setminus h_k(I_k))>0$ for some $k\in\{1\ldots,d\}$, or $\mu(h_j(I_j))>0$ for some $j\in E(f)$. In both cases, there exists $Q$ in $\mathcal{P}\vee f^{-1}\mathcal{P}$, contained in $I_k\setminus h_k(I_k)$ or $h_j(I_j)$ as appropriate, such that $\mu(Q)>0$ and $Df(x)>1$ for every $x\in Q$.

Let $\Sigma$ be the set of all infinite words in the alphabet $\N$ and for every $n$ in $\N$ put 
\begin{equation*}
    \Sigma_n\=\N^n \text{ and } \Sigma^*\=\bigcup_{n\in\N} \N^n.
\end{equation*}
For every $n$ in $\N$ and every $\underline{\ell}$ in $\Sigma_n$ the \emph{length} of $\underline{\ell}$ is $n$ and it is denoted by $|\underline{\ell}|$. An infinite sequence of pairwise distinct functions $(\phi_\ell)_{\ell\in\N}$ from $Q$ into $Q$ is called an \emph{iterated function system (IFS)}. For every $n$ in $\N$ and every finite word $\ell_1\cdots\ell_n$ in $\Sigma_n$ put 
\begin{equation}
    \phi_{\ell_1\cdots \ell_n} \= \phi_{\ell_1}\circ\cdots \circ\phi_{\ell_n}.
\end{equation}
We say that the IFS is \emph{free} if for all $\underline{\ell}$ and $\underline{\ell}'$ in $\Sigma^*$ with $\underline{\ell}\neq \underline{\ell}'$ we have that $\phi_{\underline{\ell}}$ is different from $\phi_{\underline{\ell}'}$. We say that the IFS is \emph{generated} by $f$ if for every $\ell$ in $\N$ there is $m_{\ell}$ in $\N$ such that $f^{m_\ell}\circ \phi_{\ell}$ is the identity on  $Q$. We say that $(m_\ell)_{\ell\in\N}$ is the \emph{time sequence} of $(\phi_{\ell})_{\ell\in\N}$. For every $n$ in $\N$ and every finite word $\ell_1\cdots\ell_n$ in $\Sigma_n$ put 
\begin{equation}
    m_{\ell_1\cdots\ell_n}\= m_{\ell_1}+\cdots m_{\ell_n}.
\end{equation}
We say that $(\phi_{\ell})_{\ell\in\N}$ is \emph{hyperbolic with respect to} $f$, if there are constants $C>0$ and $\lambda>1$ such that for every $x\in Q$, for every $\underline{\ell}\in\Sigma^*$ and for every $j\in\{1,\ldots,m_{\underline{\ell}}\}$ one has
\begin{equation}\label{eq hyperbolic IFS}
    |Df^{j}(f^{m_{\underline{\ell}}-j}(\phi_{\underline{\ell}}(x)))|\geq C\lambda^j.
\end{equation}
\begin{prop}\label{prop key lemma}
    For each branch Hölder-continuous potential $\varphi:[0,1]\to \R$ of exponent $\gamma\in(0,1]$, there are a constant $C>0$ and a free hyperbolic IFS $(\phi_{\ell})_{\ell\in \N}$ generated by $f$ with strictly increasing time sequence $(m_\ell)_{\ell\in\N}$ such that
    \begin{equation}\label{eq prop key lemma}
        \inf_{z\in Q} S_{m_\ell}\varphi(\phi_\ell(z)) \geq m_\ell \int \varphi \dd\mu - C.
    \end{equation}
\end{prop}
\noindent
\subsection{Proof of the Key lemma assuming Proposition \ref{prop key lemma}}
 Let $(\phi_{\ell})_{\ell\in \N}$ be the IFS given by Proposition \ref{prop key lemma} with time sequence $(m_\ell)_{\ell\in \N}$ and constant $C$. Fix $z_0 \in Q$. For every $N\in \N$ put
    \begin{equation}
      \Lambda_N\=\sum_{\underline{\ell}\in \Sigma^*, m_{\underline{\ell}}=N} \exp(S_N\varphi (\phi_{\underline{\ell}}(z_0))).
    \end{equation}
   Observe that, since the IFS $(\phi_{\ell})_{\ell\in \N}$ is free and generated by $f$, for two distinct words $\underline{\ell}$ and $\underline{\ell}'$ in $\Sigma^*$ with $m_{\underline{\ell}}=m_{\underline{\ell}'}$, the images of $\phi_{\underline{\ell}}$ and $\phi_{\underline{\ell}'}$ are disjoint, as both $\phi_{\underline{\ell}}$ and $\phi_{\underline{\ell}'}$ are inverse branches of the same iterate of $f$. Then by Lemma \ref{tree pressure} we have
    \begin{equation}
        P(\varphi) \geq \limsup_{N\to+\infty}\frac{1}{N}\log \Lambda_N.    
    \end{equation}
    Now consider the following generating function
    \begin{equation}
        \Xi(s) \= \sum_{N\in\N} \Lambda_N s^N.
    \end{equation}
    The radius of convergence R of $\Xi(s)$ is at least $\exp(-P(\varphi))$. Observe that
    \begin{equation}
        \Xi(s)= \sum_{\underline{\ell}\in \Sigma^*}\exp( S_{m_{\underline{\ell}}} \varphi (\phi_{\underline{\ell}}(z_0)))s^{m_{\underline{\ell}}}
    \end{equation}
    
    By Proposition \ref{prop key lemma} for every $\underline{\ell}\in\Sigma^*$ we have
    \begin{equation*}
        S_{m_{\underline{\ell}}}
        \varphi(\phi_{\underline{\ell}}(z_0)) \geq m_{\underline{\ell}} \int \varphi \dd\mu - |\underline{\ell}|C.
    \end{equation*}
    Then, defining 
    \begin{equation}
        \Phi(s) \= \sum_{\ell=1}^{+\infty} \exp\left( m_\ell \int \varphi \dd \mu - C\right)s^{m_\ell}
    \end{equation}
    we get that the power series in $s$
    \begin{equation}\label{sum of Phi}
        \Phi(s) + \Phi(s)^2 + \Phi(s)^3+\cdots
    \end{equation}
    has coefficients smaller than or equal to the corresponding coefficients of $\Xi(s)$. Observe that, since $(m_{\ell})_{\ell\in \N}$ is strictly increasing, the radius of convergence of $\Phi(s)$ is $\widehat{R}= \exp(-\int \varphi \dd\mu)$ and that 
    \begin{equation}
        \lim_{s\to\widehat{R}^-} \Phi (s) = +\infty.
    \end{equation}
    Then, there is $s_0\in (0,\widehat{R})$ such that $\Phi(s_0)\in (1,+\infty)$, and this implies that the radius of convergence of the series \eqref{sum of Phi} is strictly smaller than $s_0$. Then
    \begin{equation*}
        \exp(-P(\varphi)) \leq R < s_0 <\widehat{R} = \exp\left(-\int \varphi \dd \mu \right),
    \end{equation*}
    finishing the proof of the lemma.

\subsection{Proof of Proposition \ref{prop key lemma}}
    \begin{lemm}\label{lemma A2}
        Let $(z_n)_{n\in\N_0}$ be a sequence in $[0,1]$ such that $z_0$ is in $Q$ and for every $n$ in $\N_0$ we have that $z_n = f(z_{n+1})$. Let $M\geq 3$ be an integer and let $(n_\ell)_{\ell\in \N}$ be a strictly increasing sequence of positive integers such that $n_{\ell+1}\geq n_\ell + M$. For every $\ell$ in $\N$, let $x_\ell$ be a point of $Q$ in $f^{-M}(z_{n_\ell})$ different from $z_{n_\ell + M}$, and let $\phi_\ell$ be the inverse branch of $f^{n_\ell + M}$ from $Q$ into itself such that $\phi_\ell(z_0)=x_\ell$. Then the IFS $(\phi_\ell)_{\ell\in\N}$ generated by $f$ is free with time sequence $(n_\ell + M)_{\ell\in \N}$.
    \end{lemm}
    \begin{proof}
        Let $\underline{\ell}= \ell_1\cdots\ell_n$ and $\ell^{'}= \ell_1^{'}\cdots \ell_k^{'}$ be in $\Sigma^*$ with $\underline{\ell}\neq \underline{\ell}^{'}$. Assume that $m_{\ell} \neq m_{\underline{\ell}^{'}}$. Without loss of generality we assume that $m_{\underline{\ell}}<m_{\underline{\ell}^{'}}$. Suppose we had $\phi_{\underline{\ell}}=\phi_{\underline{\ell}^{'}}$. Then 
        $f^{m_{\underline{\ell}^{'}}}\circ \phi_{\underline{\ell}}= \text{Id}|_{Q}$, which implies that $f^{m_{\underline{\ell}^{'}}-m_{\underline{\ell}}}= \text{Id}|_{Q}$ and thus, $m_{\underline{\ell}}= m_{\underline{\ell}^{'}}$ giving a contradiction. Now assume that $m_{\underline{\ell}}=m_{\underline{\ell}^{'}}$. We also assume that $\ell_n\neq \ell_k^{'}$; the general case can be reduced to this one. Without loss of generality, we assume that $\ell_n<\ell_k^{'}$. In particular, we have $m_{\ell_n}<m_{\ell_k^{'}}$. Suppose we had $\phi_{\underline{\ell}}= \phi_{\underline{\ell}^{'}}$ then $f^{m_{\underline{\ell}}-m_{\ell_n}}\circ \phi_{\underline{\ell}}= f^{m_{\underline{\ell}^{'}}-m_{\ell_n}}\circ \phi_{\underline{\ell}^{'}}$ and thus,
        \begin{equation*}
            \phi_{\ell_n}= f^{m_{\ell_k^{'}}-m_{\ell_n}}\circ \phi_{\ell_k^{'}}= f^{n_{\ell_k^{'}}-n_{\ell_n}}\circ \phi_{\ell_k^{'}}.
        \end{equation*}
        Evaluating this last equality at $z_0$ and using that $n_{\ell_k^{'}}-n_{\ell_n}>M$ we get that
        \begin{equation*}
            x_{\ell_n}= f^{n_{\ell_k^{'}}-n_{\ell_n}}(x_{n_{\ell_k^{'}}})= f^{n_{\ell_k^{'}}-n_{\ell_n}-M}(z_{n_{\ell_k^{'}}}) = z_{n_{\ell_n}+M}.
        \end{equation*}
        However, by hypothesis, these two points are different. Therefore, $\phi_{\underline{\ell}}\neq \phi_{\underline{\ell^{'}}}$ and this concludes the proof of the lemma.
    \end{proof}
    \begin{lemm}\label{lemma birkhof sums key lemma}
        Let $\gamma\in (0,1]$ and $\varphi\in BH^\gamma(f)$. For each IFS $(\phi_\ell)_{\ell\in\N}$ generated by f, hyperbolic with respect to $f$ and with time sequence $(m_\ell)_{\ell\in\N}$ the following holds. There is a constant $\Delta'>0$ such that for every $\underline{\ell}\in \Sigma^*$ and all $x$ and $y$ in $Q$ we have 
        \begin{equation}
            |S_{m_{\underline{\ell}}} \varphi(\phi_{\underline{\ell}}(x)) 
            -S_{m_{\underline{\ell}}}\varphi(\phi_{\underline{\ell}}(y))|  \leq \Delta'.
        \end{equation}
    \end{lemm}
    \begin{proof}
        Let $N\in\N$ be such that $\varphi|_P$ is Hölder-continuous of exponent $\gamma$ for every $P$ in $\bigvee_{j=0}^{N-1}f^{-j}\mathcal{P}$. Then, there is a constant $D>0$ such that for each $P$ in $\bigvee_{j=0}^{N-1}f^{-j}\mathcal{P}$ and each $x$ and $y$ in $P$ we have
        \begin{equation}\label{eq Holder constant key lemma}
            |\varphi(x) - \varphi(y)|\leq D|x-y|^{\gamma}.
        \end{equation}
        On the other hand, by \eqref{eq hyperbolic IFS} there are constants $C>0$ and $\lambda>1$ such that for every $j\in \{1,\ldots,m_{\underline{\ell}}\}$
        \begin{equation}\label{eq contraction key lemma}
            |f^{m_{\underline{\ell}}-j}(\phi_{\underline{\ell}}(x)) - f^{m_{\underline{\ell}}-j}(\phi_{\underline{\ell}}(y)) |\leq C\lambda^{-j}.
        \end{equation}
        On the one hand, if $m_{\underline{\ell}}\leq N$ we have
        \begin{equation}
            |S_{m_{\underline{\ell}}} \varphi(\phi_{\underline{\ell}}(x)) 
            -S_{m_{\underline{\ell}}}\varphi(\phi_{\underline{\ell}}(y))| \leq 2N\sup|\varphi|.
        \end{equation}
        On the other hand, if $m_{\underline{\ell}}>N$ then from \eqref{eq Holder constant key lemma} and \eqref{eq contraction key lemma} we obtain
        \begin{equation*}
        \begin{split}
            |S_{m_{\underline{\ell}}} \varphi(\phi_{\underline{\ell}}&(x)) 
            -S_{m_{\underline{\ell}}}\varphi(\phi_{\underline{\ell}}(y))|\\&\leq \sum_{j=0}^{m_{\underline{\ell}}-1-N}|\varphi(f^j(\phi_{\underline{\ell}}(x))) - \varphi(f^j(\phi_{\underline{\ell}}(y)))| + \sum_{j=m_{\underline{\ell}}-N}^{m_{\underline{\ell}}-1}|\varphi(f^j(\phi_{\underline{\ell}}(x))) - \varphi(f^j(\phi_{\underline{\ell}}(y)))|\\
            &\leq D\sum_{j=0}^{m_{\underline{\ell}}-1-N}|f^{j}(\phi_{\underline{\ell}}(x)) - f^j(\phi_{\underline{\ell}}(y))|^{\gamma} + 2N\sup |\varphi|\\
            &\leq DC^\gamma \sum_{j=0}^{m_{\underline{\ell}-1-N}} (\lambda^{\gamma})^{-(m_{{\underline{\ell}}}-j)} + 2N\sup|\varphi|.
            \end{split}
        \end{equation*}
        Taking $\Delta^{'} \= DC^{\gamma} \sum_{j=0}^{+\infty} (\lambda^{\gamma})^{-j}  + 2N\sup |\varphi|$ we finish the proof of the lemma.
    \end{proof}
    \begin{lemm}\label{lemma 4 key lemma}
        Let $(X,\sB,\mu)$ be a probability space and let $T\colon X\to X$ be an ergodic measure preserving map. Then, for each integrable function $\varphi\colon X\to\R$ with null integral, there is a full measure set of $x$ in $X$ such that
        \begin{equation}
            \limsup_{n\to+\infty} S_n\varphi(x) \geq 0.
        \end{equation}
    \end{lemm}
    \begin{proof}
        See \cite[Lemma A.4]{coronelRivera1}.
    \end{proof}
    Put $I\=[0,1]$ and recall that $\mu$ is an ergodic measure in $\mathcal{M}_f\setminus \mathcal{S}$  such that $\mu(Q)>0$. Denote by $(\widehat{I},\widehat{f})$ the natural extension of $(I,f)$. That is, $\widehat{I}$ is the set of sequences $(z_n)_{n\in\N_0}$ in $I$ such that for every $n\in \N_0$ one has $z_n=f(z_{n+1})$, and $\widehat{f}$ is the bijective map from $\widehat{I}$ onto $\widehat{I}$ defined for every $(z_n)_{n\in\N_0}$ in $\widehat{I}$ by
    \begin{equation}
        \widehat{f}((z_n)_{n\in\N_0})= (f(z_0),z_0,z_1,z_2,\ldots,z_n,\ldots).
    \end{equation}
    The space $\widehat{I}$ inherits a Borel $\sigma$-algebra as a subset of the product space $\prod_{n=0}^{+\infty}I$ and the map $\widehat{f}$ is a measurable isomorphism. Denote by $\Pi\colon \widehat{I}\to I$ the projection onto the zeroth coordinate. We have that $\Pi\circ\widehat{f}= f\circ \Pi$, and there is a unique invariant probability measure $\nu$ for $\widehat{f}$ such that $\Pi_*\nu = \mu$. Since $\mu$ is ergodic, the measure $\nu$ is also ergodic for $\widehat{f}$ and $\widehat{f}^{-1}$. Since $\mu(Q)>0$ we have that $\nu(\Pi^{-1}(Q))>0$. Put $\widehat{I}'\=\Pi^{-1}(Q)$. Fix $\underline{z}$ in $\widehat{I}'$ such that $\underline{z}$ is generic for the Ergodic Theorem for $\widehat{f}^{-1}$, $\nu$ and the bounded measurable function $\log Df\circ \Pi$. Then, we have
    \begin{equation}\label{eq generic erg thm}
        \lim_{n\to+\infty}\frac{1}{n}\sum_{j=0}^{n-1}\log Df\circ \Pi(\widehat{f}^{-j}\underline{z})= \chi_\mu(f) >0.
    \end{equation}
    By Lemma \ref{lemma 4 key lemma} applied to the function $\varphi\circ \Pi\circ f^{-1} - \int\varphi d\mu$, we can choose $\underline{z}$ so that, in addition, we have
    \begin{equation}
        \limsup_{n\to+\infty} \sum_{j=0}^{n-1}\left(\varphi \circ\Pi\circ\widehat{f}^{-1}(\widehat{f}^{-j}\underline{z}) - \int \varphi \dd\mu\right)\geq 0.
    \end{equation}
    For every $\ell\in \N$ there is a unique inverse branch $\widetilde{\phi}_\ell$ of $f^{\ell}$ verifying that $\Pi(\widehat{f}^{-\ell}\underline{z})$ is in the image of $\widetilde{\phi}_\ell$. Recall that either $Q\subseteq I_k\setminus h_k(I_k)$ for some $k\in\{1,\ldots,d\}$ or $Q\subseteq h_j(I_j)$ for some $j\in E(f)$. In the first case, $Q$ is the image of the composition of two distinct inverse branches of $f$, and in the second case, we have that $\overline{Q} = h_j^2([0,1])$. Thus, in all the cases, by Lemma \ref{lemma bound Df^n} and \eqref{eq generic erg thm} there are $C'>0$ and $\lambda>1$ such that, increasing $C'>0$ if necessary, for every $n\in \N$ and every $z\in Q$ we have
    \begin{equation}\label{eq key lemma expanding}
        Df^n(\widetilde{\phi}_n(z))\geq C' \lambda^n.
    \end{equation}
    Then there is a strictly increasing sequence $(n_\ell)_{\ell\in \N}$ in $\N$ such that
    \begin{equation}\label{eq n_ell}
        S_{n_\ell} \varphi(z_{n_\ell})\geq n_\ell \int \varphi \dd\mu - 1, n_{\ell+1}\geq n_{\ell}+4, C'\lambda^{\frac{n_1}{2}}>\lambda^{2},
    \end{equation}
and the sequence $(z_{n_\ell})_{\ell\in\N}$ converges to some $ w \in [0,1]$.
    Observe that $Q$ contains $d$ elements of $\mathcal{P}\vee f^{-1}\mathcal{P}\vee f^{-2}\mathcal{P}$, let say $Q_1,\ldots,Q_d$, and let $a_1<\ldots <a_{d+1}$ be such that 
    \begin{equation}
        \text{int } Q_j = (a_j,a_{j+1}) \text{ for every $j\in\{1,\ldots,d\}$}.
    \end{equation}
    Then we have that $f^3$ maps $(a_j,a_{j+1})$ bijectively onto $(0,1)$ for every $j\in\{1,\ldots,d\}$. For each $j\in\{1,\ldots,d\}$, let $Q_j^1,\ldots,Q_j^d$ be the elements of $\mathcal{P}\vee f^{-1}\mathcal{P}\vee f^{-2}\mathcal{P} \vee f^{-3}\mathcal{P}$ contained in $Q$ ordered increasingly. Let $a'$ and $a''$ be such that
    \begin{equation}
        \text{int } Q_1^d = (a',a_2) \text{ and }\text{int }Q_2^1=(a_2,a''). 
    \end{equation}
    Then $f^4$ maps $(a',a_2)$ and $(a_2,a'')$ onto $(0,1)$ bijectively, and notice that $a'<a_2<a''$. Now we distinguish two cases: The first case is when $w\in[0,a_2]$. By passing to a subsequence, we can assume that
    \begin{equation}
        (z_{n_\ell})_{\ell\in\N} \text{ is in } (0,a'').
    \end{equation}
    For every $\ell\in\N$ we put
    \begin{equation}
        \phi_\ell:= (f^4|_{Q_2^2})^{-1}\circ (\widetilde{\phi}_{n_\ell}|_Q).
    \end{equation}
    The second case is when $w\in (a_2,1]$. Again, by passing to a subsequence we can assume that 
    \begin{equation}
        (z_{n_\ell})_{\ell\in\N} \text{ is in }(a',1].
    \end{equation}
    For every $\ell\in\N$ we put
    \begin{equation}
        \phi_\ell:= (f^4|_{Q_1^{d-1}})^{-1}\circ (\widetilde{\phi}_{n_\ell}|_Q).
    \end{equation}
In both cases put $m_\ell\=n_\ell +4$. We have in all the cases that for every $\ell\in\N$, the point $x_\ell \= \phi_\ell(z_0)$ is in $f^{-4}(z_{n_\ell})$ and it is different from $z_{m_\ell}$. Thus by Lemma \ref{lemma A2} and the second inequality in \eqref{eq n_ell}, we get that the IFS $(\phi_\ell)_{\ell\in\N}$ is free.

Now, we prove that the IFS $(\phi_\ell)_{\ell\in\N}$ is hyperbolic. By 
\eqref{eq key lemma expanding}, the second and the third inequalities in \eqref{eq n_ell}, for every $z$ in $Q$ and for every $j\in\{0,1,2,3\}$ we have in all the cases that 
\begin{equation}
    Df^{m_\ell -j}(f^j(\phi_\ell(z))) \geq \lambda^{\frac{m_\ell-j}{2}}.
\end{equation}
Together with \eqref{eq key lemma expanding}, this implies that for every $j\in\{1,\ldots,m_{\underline{\ell}}\}$ one has
\begin{equation}
    Df^j(f^{m_{\underline{\ell}}-j}(\phi_{\underline{\ell}}(z)))\geq C'\lambda ^{\frac{j}{2}},
\end{equation}
and thus, the IFS $(\phi_{\ell})_{\ell\in\N}$ is hyperbolic with constants $C'>0$ and $\lambda^{1/2}>1$.

It remains to prove \eqref{eq prop key lemma}. Put
\begin{equation}
    C_1\= - \inf_{x\in[0,1]} \varphi \text{ and } C''\= 1 +4C_1 + 4 \int \varphi \dd\mu.
\end{equation}
By the first inequality in \eqref{eq n_ell}, we have
\begin{equation}
\begin{split}
    S_{m_\ell} \varphi (\phi_\ell(z_0)) &= S_4 \varphi (\phi_\ell(z_0)) + S_{n_\ell} \varphi(\widetilde{\phi}_{n_\ell}(z_0))\\
    &\geq -4C_1 + n_\ell \int \varphi \dd\mu - 1\\
    &\geq m_\ell \int \varphi \dd\mu - C''.
    \end{split}
\end{equation}
Together with Lemma \ref{lemma birkhof sums key lemma}, this completes the proof of \eqref{eq prop key lemma} with $C= C'' + \Delta'$, concluding the proof of Proposition \ref{prop key lemma}.

\section*{Acknowledgments}

We would like to thank Juan Rivera-Letelier and Godofredo Iommi for carefully reading the manuscript and for their valuable comments and corrections.


\end{document}